%% file: text_nieuw_seeds.tex
\documentclass[final, nomarks]{dmtcs-episciences} 
\usepackage[utf8]{inputenc} 

\usepackage{geometry} 
\usepackage{amsmath}
\usepackage{graphicx} 

\usepackage{booktabs} 
\usepackage{array} 
\usepackage{paralist} 
\usepackage{verbatim} 
\usepackage{subfig} 
\usepackage{tikz}
\usepackage{tkz-euclide}
\usepackage{fancyhdr} 

\usepackage{amsfonts}

\usepackage{amsmath}

\newtheorem{theorem}{Theorem $\!\!$}[section]
\newtheorem{lemma}[theorem]{Lemma $\!\!$}
\newtheorem{remark}[theorem]{Remark $\!\!$}
\newtheorem{corollary}[theorem]{Corollary $\!\!$}

\title[Uniquely hamiltonian graphs]{ Uniquely hamiltonian graphs for many sets of degrees }
\author[G. Brinkmann and M. De Pauw]{Gunnar Brinkmann \and Matthias De Pauw }
\affiliation{Ghent University, Ghent, Belgium}


\begin{document}

\publicationdata{vol. 26:3}{2024}{7}{10.46298/dmtcs.13129}{2024-02-27; 2024-02-27; 2024-11-06}{2024-09-07}

\date{}

\maketitle

\begin{abstract}

  \bigskip

  We give constructive proofs for the existence of uniquely hamiltonian graphs for various sets of degrees. We give constructions for all sets with minimum $2$
  (a trivial case added for completeness), all sets with minimum $3$ that contain an even number (for sets without an even number it is known that no uniquely
  hamiltonian graphs exist), and all sets with minimum $4$, except $\{4\}, \{4,5\}$, and $\{4,6\}$. For minimum degree
  $3$ and $4$, the constructions also give $3$-connected graphs.

  We also introduce the concept of seeds, which makes the above results possible and might be useful in the study of Sheehan's conjecture. Furthermore, we
  prove that $3$-connected
  uniquely hamiltonian $4$-regular graphs exist if and only if $2$-connected
  uniquely hamiltonian $4$-regular graphs exist. 
  \end{abstract}

\section{Introduction}

The most important problem for hamiltonian cycles is of course which properties guarantee the existence of a hamiltonian cycle, but as soon as the existence of
a hamiltonian cycle is known, the question arises how many hamiltonian cycles exist. In \cite{fewHC}, recent results and an overview of older results on graphs
with few hamiltonian cycles are given. The extremal case is when a graph contains a single hamiltonian cycle, that is: it is {\em uniquely hamiltonian}.  A
crucial role for the existence of a uniquely hamiltonian graph is played by the combination of vertex degrees present in the graph. Already in 1946 Tutte
reported a result by Smith that uniquely hamiltonian cubic graphs don't exist \cite{nocubicUH}.  A long standing conjecture by Sheehan \cite{sheehan} states
that this should in fact be the case for all $d$-regular graphs with $d>2$.  The result by Smith was later improved by Thomason \cite{Thomason_odd} showing that
uniquely hamiltonian graphs where all vertices have odd degree don't exist.  In \cite{noUH23} it is shown that no $d$-regular uniquely hamiltonian graphs exist
if $d\ge 23$.  So while there are e.g.\ neither uniquely hamiltonian graphs with all degrees 3 nor with all degrees 24, a special case of what we will prove
will be that there are uniquely hamiltonian graphs if both these vertex degrees are allowed. For even $d$ with $4\le d \le 22$ it is not known whether
$d$-regular uniquely hamiltonian graphs exist.  In \cite{4_14} Fleischner shows that there are uniquely hamiltonian graphs with minimum degree 4. He constructs
graphs with vertices of degree $4$ and $14$ and graphs where {\em the maximum degree can grow even larger} -- without specifying which degrees can occur. We
will use an improved version of his method to prove that for all sets $M$ with minimum $4$, except maybe for $\{4\}, \{4,5\}$, and $\{4,6\}$,
uniquely hamiltonian graphs exist, so that the set of vertex degrees is exactly $M$.  Furthermore we characterize sets of degrees with minimum $2$ or $3$ for
which uniquely hamiltonian graphs exist completely.

The term {\em graph} always refers to a simple undirected graph, that is: without multiple edges and without loops. If multiple edges are allowed, we use the term {\em multigraph}. Loops
are never allowed, as they are trivial in the context of uniquely hamiltonian graphs.

We define the degree set $M_{deg}(G)$ of a graph (or multigraph) $G$ with vertex set $V$ as  $M_{deg}(G)=\{\deg(v)\mbox{ } |  v\in V\}$.

For a set $M=\{d_0,d_1,d_2,\dots ,d_k\}$ with $d_0< d_1< \dots < d_k$, we say that a $2$-connected (if $2\in M$), resp.\ $3$-connected (otherwise)
uniquely hamiltonian graph $G$ {\em realizes}
$M$ if $M_{deg}(G)=M$. If such a $G$ exists, we define $M$ to be {\em uhc-realizable}.

Next to the question whether a set $M$ is uhc-realizable, it is also interesting which role is played by the larger degrees. Our emphasis is on the smallest
degree $d_0$ and we want to know whether the number of times that the degrees $d_1,\dots ,d_k$ occur can be bounded by a constant even for very large graphs, so
that the average degree can be arbitrarily close to the smallest degree. On the other hand it might also be interesting to know, whether the larger degrees can
occur an unbounded number of times and maybe also occur for at least a fixed fraction of the vertices also in arbitrarily large graphs.  The average degree
would in that case be bounded from below by the minimum degree times a constant factor $c>1$.  The strongest requirement is, if both can occur and even in
combination depending on the $d_i$. We formalize that by the following definition:

For a set $M=\{d_0,d_1,d_2,\dots ,d_k\}$ with $k>0$, $d_0< d_1< \dots < d_k$ we say that $M$ is {\em strongly uhc-realizable}, if for
each partition $D_1, D_2$ of $\{d_1,\dots,d_k\}$ (with one of $D_1,D_2$ possibly empty) there are constants
  $c_1\in \mathbb{N},c_2\in \mathbb{R}$, $c_2>0$, and an infinite sequence of graphs $G_i=(V_i,E_i)$ realizing $M$, so that for all $d\in D_1$ each $G_i$ has at most
$c_1$ vertices of degree $d$, and for each $d'\in D_2$ each $G_i$ has at least $c_2|V_i|$ vertices of degree $d'$.

\section{Minimum degree 2 or 3}

We will start with an easy remark that is mainly contained for completeness:

\begin{remark}\label{rem:2}

  Any finite set $M=\{d_0=2,d_1,d_2,\dots ,d_k\}\subset \mathbb{N}$ with $2< d_1 < d_2 < \dots < d_k$ is uhc-realizable and if $k>0$, it is also strongly uhc-realizable.

\end{remark}

\begin{proof}

We will first prove that $M$ is uhc-realizable. $|M|=1$ is trivial. If $|M|=2$, one can take $K_{d_1+1}$ and subdivide the edges of a hamiltonian
cycle.  If $|M|>2$ one can e.g. take complete graphs $K_{d_1+1}, \dots K_{d_k+1}$, remove an edge $e_{d_i}\in K_{d_i+1}$ for $1\le i<k$, an edge
$e'_{d_i}\in K_{d_i+1}$ for $2\le i\le k$ with $e_{d_i}\cap e'_{d_i}=\emptyset$ for $2\le i < k$, and then connect the endpoints of $e_{d_i}$ and $e'_{d_{i+1}}$
for $1\le i< k$. The result is obviously hamiltonian and 2-connected and after subdividing the edges of a hamiltonian cycle, one has a uniquely hamiltonian graph with
exactly the vertex degrees in $M$.

To show that $M$ is strongly uhc-realizable, assume a partition $D_1,D_2$ to be given.  If $D_2=\emptyset$ one can subdivide edges on the hamiltonian cycle
arbitrarily often to obtain the sequence of graphs. If $D_2\not=\emptyset$ one can use the above construction for multisets $M'_j$ containing the same elements as $M$,
but numbers in $D_1$ exactly once and numbers in $D_2$ exactly $j$ times.

  \end{proof}

\section{Minimum degree 3 and 4}

The following construction is a slight modification of a construction by H.~Fleischner \cite{4_14}.

Let $P=(V,E)$ be a graph and $s,t,v\in V$ be vertices.
If there is a unique hamiltonian path from $s$ to $t$ in the graph $P_{-v}=P[V\setminus\{v\}]$
induced by $V\setminus\{v\}$, we call ${\cal P}=(P,s,t,v)$ a {\em weak  H-plugin} or just an  H-plugin. If in addition there is no hamiltonian path from $s$ to $t$ in $P$ (so also containing $v$),
we call ${\cal P}=(P,s,t,v)$ a {\em strong H-plugin}.

In cases where $s,t,$ and $v$ are clear from the context, we will also refer to the graph $P$ alone as an H-plugin.

For an H-plugin $(P,s,t,v)$ and a graph $G$ with a
vertex $x$ of degree $3$ with neighbours $y,x_1,x_2$, 
we define the {\em {\cal P}-splice} of $G$ at $\{x,y\}$, denoted as $O(x,y,{\cal P})$ as the graph obtained by removing $x$, connecting $x_1$ with the vertex $s$ in a copy $P'$ of $P$, $x_2$ with the vertex $t$ in
$P'$
and identifying the vertex $v$ in $P'$ with $y$. This operation is sketched in Figure~\ref{fig:plugin}. We will also refer to it shortly as {\em splicing the edge $\{x,y\}$}. The notation  $O(x,y,{\cal P})$ does not take into account which of the vertices is $x_1$ and which is $x_2$, so in general
 $O(x,y,{\cal P})$ is one of the two possibilities. Elementary arguments  show that if
$P$ -- or at least $P$ together with a new vertex connected to $s,t$, and $v$ -- as well as 
$G$ are 3-connected, then $O(x,y,{\cal P})$ is 3-connected.

\begin{figure}[tb]
	\centering
	\includegraphics[width=0.7\textwidth]{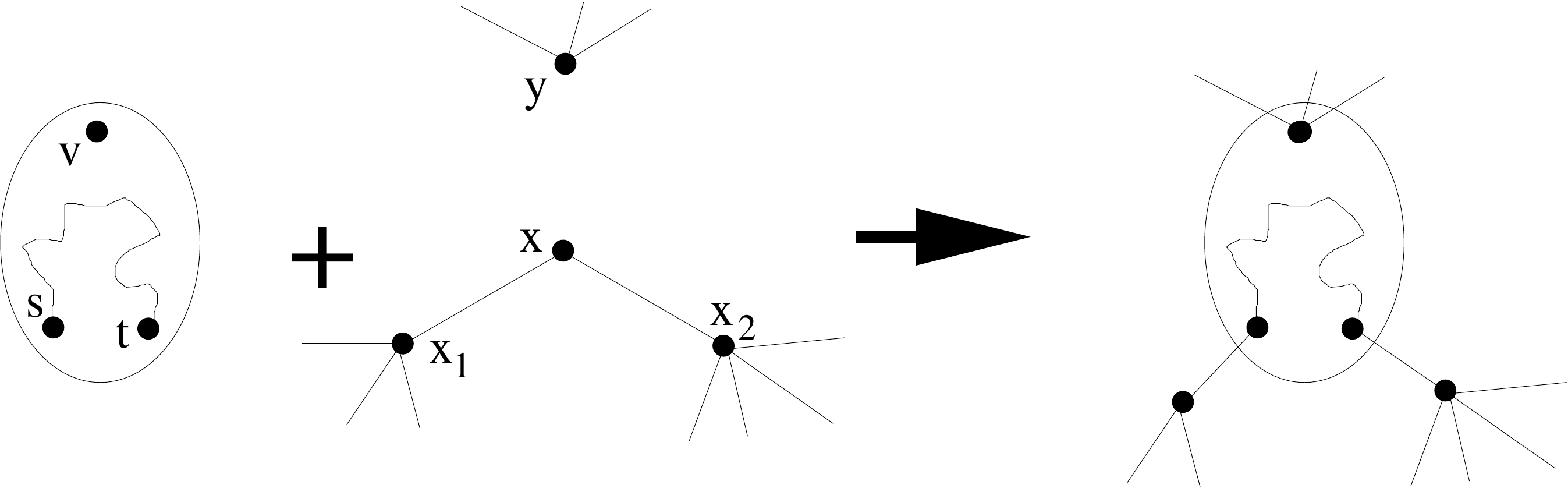}
	\caption{The splicing operation.}
	\label{fig:plugin}
\end{figure}

The following lemma and corollary are stronger versions of Lemmas 1,2, and 3 in \cite{4_14}.

\begin{lemma}\label{lem:op}(parts already in \cite{4_14})
  Let $G=(V,E)$ be a graph with a unique hamiltonian cycle $C_H$, $x\in V$ of degree $3$ with neighbour $y$, so that the edge $\{x,y\}$ is not on $C_H$. Let  ${\cal P}=(P,s,t,v)$ be an H-plugin.

  If at least one of the following three conditions is fulfilled, then $O(x,y,{\cal P})$ has a unique hamiltonian cycle $C_{H,O}$. Except for the edges incident
  with $x$, all edges of $C_H$ are also contained in $C_{H,O}$.

  \begin{description}
  \item[(i)] $G[V\setminus \{y\}]$ is not hamiltonian.
  \item[(ii)] $\{x,y\}$ lies in a triangle.
  \item[(iii)] ${\cal P}$ is a strong H-plugin.
  \end{description}

\end{lemma}

Condition (iii) also explains the name {\em strong} H-plugin: while in general the splicing of edges that are not on the unique hamiltonian cycle
only guarantees a unique hamiltonian cycle in the result if the edges satisfy some extra condition, this extra condition is not necessary if ${\cal P}$ is strong.

\begin{proof}
  As $s$ and $t$ have only one edge to the outside of (the copy of) $P$ in $O(x,y,{\cal P})$, none of them can be incident only with edges of a hamiltonian cycle $C_{H,O}$ of $O(x,y,{\cal P})$ that lie outside $P$.
  To this end there are in principle three ways how $C_{H,O}$  could pass through $P$:

  \begin{description}
  \item[a.)] by a hamiltonian path of $P_{-v}$  from $s$ to $t$ while the vertex $v=y$ is incident to two edges of $C_{H,O}$ not in $P$,
  \item[b.)] by a hamiltonian path of $P$ from $v$ to $s$ or to $t$,
  \item[c.)] by a hamiltonian path of $P$ from $s$ to $t$.
  \end{description}

  In all three cases (i), (ii), and (iii) of the lemma, we can get a hamiltonian cycle of $O(x,y,{\cal P})$ passing $P$ like described in a.)\ if we replace the part $x_1,x ,x_2$ in $C_H$ by
  $x_1,s , \dots ,t ,x_2$ with the middle part the unique hamiltonian path from $s$ to $t$ in $P_{-v}$. So there is always a hamiltonian cycle for case a.), but that cycle is
  unique due to the two paths in $P_{-v}$ and outside $P_{-v}$ being unique. 

  Assume now that $O(x,y,{\cal P})$ has a hamiltonian cycle passing $P$ as in case b.) and assume w.l.o.g.\ that the endpoint is $s$. Replacing the part $y=v,\dots ,s ,x_1$
  by $y,x,x_1$, we get a hamiltonian cycle of $G$ containing $\{x,y\}$, which does not exist, as $C_H$ is unique. So a hamiltonian cycle falling into case b.)\ does not exist.

  It remains to be shown that also case c.) can not occur under the additional prerequisites.

  {\bf (i)}

  Assume that $O(x,y,{\cal P})$ has a hamiltonian cycle passing $P$ as in case c.). Replacing the part $x_1,s ,\dots ,t ,x_2$ (now also containing
  $v=y$) by $x_1,x ,x_2$, we get a cycle in $G$ missing only $y$ -- that is: a hamiltonian cycle of $G[V\setminus \{y\}]$, which does by assumption not
  exist.

  {\bf (ii)}
  This is a special case of (i).  Assume that  $G[V\setminus \{y\}]$ contains a hamiltonian cycle $C'_H$. Then $C'_H$ passes $x$ by $x_1,x ,x_2$, but replacing this
  part by $x_1, y, x, x_2$ or $x_1, x, y, x_2$ -- depending on whether the triangle is $x_1,x,y$ or $x_2,x,y$ -- we get a hamiltonian cycle of $G$ containing $e$, which does not exist, as $C_H$ is unique.

  {\bf (iii)} In this case the prerequisites are exactly that a path as in c.)\ does not exist.

  \end{proof}

\begin{corollary}\label{cor:op}(parts already in \cite{4_14})
  Let $G=(V,E)$ be a graph with a unique hamiltonian path $P_H$ from $s\in V$ to $t\in V$. Assume $x\in V$, $x\not\in \{s,t\}$ is of degree $3$ with neighbour $y$, so that the edge $\{x,y\}$ is not on $P_H$.
  Let  ${\cal P}=(P,s',t',v)$ be an H-plugin.

   If at least one of the following four conditions is fulfilled, then $O(x,y,{\cal P})$  has a unique hamiltonian path $P_{H,O}$ from $s$ to $t$. Except for the edges incident
  with $x$, all edges of $P_H$ are also contained in $P_{H,O}$.

  \begin{description}
  \item[(i)] $y\not\in \{s,t\}$, and $G[V\setminus \{y\}]$ has no hamiltonian path from $s$ to $t$.
  \item[(ii)] $\{x,y\}$ lies in a triangle.
  \item[(iii)] ${\cal P}$ is a strong H-plugin.
  \item[(iv)] $y\in \{s,t\}$.
  \end{description}

\end{corollary}

\begin{proof}

  Adding a new vertex to $G$ and connecting it with $s$ and $t$, the resulting graph $G'$ has a unique hamiltonian cycle if and only if $G$ has a unique hamiltonian
  path from $s$ to $t$. Applying Lemma~\ref{lem:op} to $G'$ we get the results. Case (iv) follows by case (i) of Lemma~\ref{lem:op}.

\end{proof}

We can now prove the main theorem for minimum degree $3$:

\begin{theorem}\label{thm:alldegcubic}

  A finite set $M=\{d_0=3,d_1,d_2,\dots ,d_k\}$ with $3< d_1 < d_2 < \dots < d_k$ of natural numbers is uhc-realizable if and only if $M$ contains an even number. In that case it
  is also strongly uhc-realizable.

\end{theorem}

\begin{proof}

  The fact that there is no uniquely hamiltonian graph $G$ with $M_{deg}(G)=M$ if $M$ contains no even number, is a well known result of
  Thomason \cite{Thomason_odd} -- no matter what the condition on connectivity is. To show that $M$ is uhc-realizable if $M$ contains an even number,
  we will explicitly construct a $3$-connected uniquely hamiltonian graph $G$ with $M_{deg(G)}=M$ in that case.

\begin{figure}[tb]
  \centering
     \resizebox{0.5\textwidth}{0.5\textwidth}
        {
          \input{one_of_smallest_uhc_3_4_colour.tikz}
          }
	\caption{The graph $U_{3,4}$ drawn as a minimum genus embedding. Sides with the same colour have to be identified. This is one of the five smallest
          uniquely hamiltonian graphs with only degrees $3$ and $4$ as given in \cite{fewHC}. The unique hamiltonian cycle is $1,2,\dots ,18$.  The vertices $3$
          and $12$ are the only vertices of degree $4$.}
	\label{fig:uhc_3_4}
\end{figure}

\begin{figure}[tb]
  \centering
       \resizebox{0.45\textwidth}{0.45\textwidth}
        {
          \input{3plus2plugin.tikz}
          }
	\caption{The Petersen graph with one edge removed gives a strong plugin $P_{3,+2}$ for $s=1$, $t=9$ and $v=10$. It can be easily checked by hand that $1,2,\dots ,9$ is the unique hamiltonian path from $s$ to $t$ if $v$ is removed and of course there is no hamiltonian path from $s$ to $t$ without removing $v$ as it would imply a hamiltonian cycle in the Petersen graph. When used as a plugin, the degrees in the copy of  $P_{3,+2}$ are $3$ and $d+2$ if the vertex identified with $v$ has degree $d$.}
	\label{fig:deg3plugin}
\end{figure}

Figure~\ref{fig:uhc_3_4} shows one of the five smallest uniquely hamiltonian graphs $G$ with $M_{deg(G)}=\{3,4\}$ (see \cite{fewHC}). By
using the strong plugin given in Figure~\ref{fig:deg3plugin} to an edge not on the hamiltonian cycle and incident to a vertex of degree $4$, we can increase the
degree of that vertex by $2$.  Doing that recursively, we can increase the degree of that vertex to any even degree. Applying the plugin to an edge incident
with two vertices of degree $3$, we can increase the degree of one of them to $5$ and recursively to any odd degree. As the number of vertices of degree $3$ can be increased
by replacing a vertex by a triangle -- and keeping the graph uniquely hamiltonian -- we can conclude that there are infinitely many (3-connected) uniquely hamiltonian graphs $G_M$ for any degree set $M=\{3,d_1,d_2,\dots ,d_k\}$
containing one or two even degrees. If we take two graphs realizing degree sets $M,M'$, remove one vertex of degree $3$ in each of them and connect the neighbours in a way that the parts of
the unique hamiltonian cycles are connected to each other, we get a graph $G_{M\cup M'}$ realizing the degree set $M\cup M'$. This way we get that for each $M=\{3,d_1,d_2,\dots ,d_k\}$ with at least
one even element there are infinitely many uniquely hamiltonian graphs $G_M$ realizing it.

Assume now that for a degree set $M=\{3,d_1,d_2,\dots ,d_k\}$ containing an even degree a partition $D_1,D_2$ of $\{d_1,d_2,\dots ,d_k\}$ is given.
There is a uniquely hamiltonian graph $G_M$ realizing $M$. If $D_2$ is empty, we can recursively replace vertices of degree $3$ by triangles to get an infinite sequence of uniquely hamiltonian
graphs realizing $M$ and having the same number of vertices of degree $d\in D_1$. If $D_2$ contains an even degree, we can make arbitrarily many copies of a graph realizing $D_2\cup \{3\}$ and recursively
combine them in the way described above with $G_M$. The result has a constant number of vertices with degree in $D_1$ and at least a constant fraction of vertices with degree in $D_2$. If finally
$D_2$ does not contain a vertex of even degree, we can recursively replace vertices of degree $3$ in $G_M$ by triangles, so that for each $k\in \mathbb{N}$ and each $d\in D_2$ we can use the plugin to
make $k$ vertices with degree $d$. As all graphs constructed in this proof are $3$-connected, this final construction proves that $M$ is strongly uhc-realizable.

\end{proof}

The repeated application of $P_{3,+2}$ does not give smallest possible graphs with this degree sequence -- in fact not even smallest graphs constructed by using plugins. There is e.g.\ a plugin on $15$
vertices increasing the degree of the identified vertex by $4$ and increasing the number of vertices by $13$ instead of $16$ when applying $P_{3,+2}$ twice.

For minimum degree $4$, it is unfortunately not so easy to give a strong plugin, but we have to construct it, starting from weak plugins.
    
We do not only want to splice one edge in a graph $G$, but each edge in some set of edges. This is in general not possible, if the edges only satisfy condition
(i) of Lemma~\ref{lem:op} or Corollary~\ref{cor:op} for $G$:
if $z$ is a vertex, so that $G[V\setminus \{z\}]$ has no hamiltonian cycle or hamiltonian path between two vertices $a,b$, it is possible that after splicing an edge $\{x,y\}$
not even close to $z$, the result $O(x,y,{\cal P})$ has a hamiltonian path or cycle in the graph with $z$ removed. If on the other hand we have
a set $E_O$ of candidate edges $\{x_1,y_1\},\dots ,\{x_k,y_k\}$ to be spliced with different $x_i$ in different triangles, or the $y_i$ are one of the starting points $s,t$ of the unique
hamiltonian path, these properties are
preserved after splicing an edge in $E_O$. This implies that in that situation we can apply the splicing operation also with a weak H-plugin to all edges
simultaneously or in any order and still draw the conclusions of Lemma~\ref{lem:op} or Corollary~\ref{cor:op}.

\bigskip

Let $G=(V,E)$ be a graph with $s,t\in V$, $v\not\in V$ and a unique hamiltonian path from $s$ to $t$. For a set $V'\subseteq V$ we define $W_{V'}(G)$ as the graph
obtained from $G$ by adding the vertex $v$ and connecting it to all vertices in $V'$ -- or formally: $W_{V'}(G)=(V_W,E_W)$ with $V_W=V\cup \{v\}$,
$E_W=E\cup \{\{v,w\} | w\in V'\}$. For a set $\{4,d_1,\dots ,d_k\}$ with $4<d_1 < d_2\dots <d_k$ we call $G=(V,E)$ a $\{4,d_1,\dots ,d_k\}$-seed, if there is a
set $V'\subseteq V$, so that if $W_{V'}(G)$ is used for splicing an edge with both endpoints of degree $3$ in a 3-connected graph, the result is $3$-connected
and the set of degrees that occur in the copy of $W_{V'}(G)$ is exactly $\{4,d_1,\dots ,d_k\}$.

\begin{remark}\label{rem:plugin}
  Let $G=(V,E)$ be a graph with a unique hamiltonian path from $s\in V$ to $t\in V$, $V'\subseteq V$ and $v\not\in V$. Then we have:

\begin{description}
  \item[(i)] $W_{V'}(G)$ is an H-plugin.
  \item[(ii)] If $V'=\{x,y\}$ and $x,y$ are the endpoints of an edge not on the unique hamiltonian path from $s$ to $t$, then $W_{\{x,y\}}(G)$ is a strong H-plugin.
\end{description}

\end{remark}

This remark follows immediately from the definitions of H-plugin and strong H-plugin and the fact that a hamiltonian path from $s$ to $t$ containing $v$ would
imply a hamiltonian path in $G$ containing the edge $\{x,y\}$.

If $G=(V,E)$ is an $M$-seed for some set $M$ and $V'\subseteq V$ is a set proving this, then the plugin $W_{V'}(G)$ is also called
an $M$-plugin.

We will use seeds to construct weak H-plugins, use those to construct strong H-plugins, and the strong H-plugins to construct uniquely hamiltonian graphs with certain sets of degrees.

We will first use the splicing operation to show how weak H-plugins imply the existence of certain strong H-plugins:

\begin{lemma}\label{lem:strong}
  
  If for a set $M=\{4,d_1,\dots ,d_k\}$ with $4<d_1<\dots <d_k\}$ there is an $M$-seed $S$, then there is a strong H-plugin ${\cal P}_{M}^{str}$, so that when  ${\cal P}_{M}^{str}$ is used
  for splicing an edge with both endpoints of degree $3$, the set of vertex degrees of the vertices in the copy of ${\cal P}_{M}^{str}$ is exactly $M$.\\
  If there are infinitely many $M$-seeds, each with for $1\le i\le k$  exactly $C_i$ vertices with degree $d_i$ when used for splicing an edge with both endpoints of degree $3$, then there
  are infinitely many strong $M$-plugins ${\cal P}_{M}^{str}$, each with $5C_i$ vertices with degree $d_i$ after splicing.

  \end{lemma}

\begin{proof}

        \begin{figure}[tb]
	\centering
        \resizebox{0.6\textwidth}{0.6\textwidth}
        {
          \input{fleischner_pmin_colour.tikz}
          }
	\caption{The graph $P^-$ from \cite{4_14}, which has  two hamiltonian cycles: $1,2,3, \dots ,15$ and $1,2,3,4,5,11,12,13,14,15,6,7,8,9,10$. As only one of them contains the edge $\{1,15\}$ it has
          a unique hamiltonian path $1,2,\dots ,15$ from $s=1$ to $t=15$. Edges with both endpoints of degree $3$ to which the splicing operation with a weak
          H-plugin can be applied while the uniqueness of the hamiltonian path is preserved, are drawn as arrows pointing at the vertices which can or must be chosen as
          $y$.  }
	\label{fig:pminus}
        \end{figure}

        Let ${\cal P}_{M}$ be the (weak) $M$-plugin constructed from $S$ as described in  Remark~\ref{rem:plugin} and assume that for $1\le i\le k$  exactly $C_i$ vertices
        in ${\cal P}_{M}$ have degree $d_i$ when it is used for splicing an edge.

        Figure~\ref{fig:pminus} shows the graph $P^-$ with a unique hamiltonian path $1,2,\dots ,15$ from $s=1$ to $t=15$ (given in \cite{4_14}).  Edges with
        both endpoints of degree $3$ to which the splicing operation with a weak H-plugin can be applied in a way that there is still a unique hamiltonian path
        between $s$ and $t$ are drawn as arrows pointing at the vertices which can or must be chosen as the vertex $y$ in the operation. If we splice these
        edges with ${\cal P}_{M}$, we get a graph with $5C_1,\dots ,5C_k$ vertices with degrees $d_1,\dots ,d_k$, $2$ vertices (the vertices $5$ and
        $11$) with degree $3$, and all other vertices with degree $4$. Due to Corollary~\ref{cor:op}, this graph still has a unique hamiltonian path from $s$ to
        $t$ not containing the edge $\{5,11\}$. If we remove the edge between $s$ and $t$, add a new vertex $v$, and connect it to the vertices $5$ and $11$,
        due to Remark~\ref{rem:plugin} we get a strong H-plugin ${\cal P}_{d}^{str}$. Each of the vertices $s$ and $t$ now has a degree $d'-1$ with $d'\in M$,
        so when applied in a
        splicing operation the degree is again $d'$. Before splicing, $v$ has degree $2$, so splicing an edge with both endpoints of degree $3$ it gets degree $4$.
        All other vertices have a degree in $M$. If we apply the
        ${\cal P}_{d}^{str}$-splice to an edge with both endpoints of degree $3$, one of them is deleted and the other one is identified with $v$ and gets degree
        $4$. If there are infinitely many H-plugins ${\cal P}_{M}$, each with $C_1,\dots ,C_k$ vertices with degrees $d_1,\dots ,d_k$, we get infinitely many strong H-plugins
        ${\cal P}_{M}^{str}$ with $5C_1,\dots ,5C_k$ vertices with degrees $4<d_1<\dots <d_k$.

\end{proof}

\begin{lemma}\label{lem:gk}

  For each $k\in \mathbb{N}$ there are $3$-connected uniquely hamiltonian graphs $G_k=(V_k,E_k)$ with $M_{deg}(G_k)=\{3,4\}$, so that the edges not on the
  hamiltonian cycle form a 2-regular subgraph containing all vertices of degree $4$ together with a matching of size at least $k$ containing all vertices of degree $3$.
  
\end{lemma}

\begin{proof}
  
  We can apply a well known technique from \cite{twotoone} to obtain a uniquely hamiltonian graph from a graph with two hamiltonian cycles that contains a cubic vertex that
  is passed by the two hamiltonian cycles in different ways.  We take two copies
  of $P^-$ and in each of them an arbitrary cubic vertex that is traversed by the two hamiltonian cycles in two different ways. Say these vertices are $v$ and
  $v'$, that the neighbours are $a,b,c$, resp.\ $a',b',c'$ and that the hamiltonian cycles pass $v$ as $a, v, b$ and $a, v, c$ (and accordingly for
  $v'$). Removing $v$ and $v'$ and adding the edges $\{a,c'\},\{b,b'\},\{c,a'\}$, only one hamiltonian cycle remains -- using the paths $a$ to $c$ in one copy
  and $c'$ to $a'$ in the other.

As in both hamiltonian cycles the vertices of degree $4$ are traversed in a way so that the edges not on the hamiltonian cycle and incident with the 4-valent
vertices form a triangle, the result will in each case have a unique hamiltonian cycle with two triangles of edges not on the hamiltonian cycle containing all
$6$ vertices of degree $4$. As each cubic vertex has exactly one edge not on the hamiltonian cycle, these edges form the required matching. Starting from this
graph, we can replace vertices of degree $3$ by triangles to increase the number of cubic vertices and therefore also the size of the matching until we have a
matching of size at least $k$.

\end{proof}

We get the following theorem as an immediate consequence:

\begin{theorem}\label{thm:set}

  Let $M \subset \{4,5,6\dots\}$ with $4\in M$ be a set, so that there are sets $M_1, M_2, \dots ,M_k$ with $\bigcup_{i=1}^k M_i = M$ and for $1\le i \le k$ there is an $M_i$-seed $S_i$.
  Then $M$ is uhc-realizable.

  If $|M|>1$, $|M_i|=2$ for $1\le i \le k$, and for each $i$ there are infinitely many $M_i$-seeds with the same number of vertices with degree different from $4$, then $M$ is also strongly uhc-realizable.
  

\end{theorem}

\begin{proof}

  Given the set $M$, we can take any uniquely hamiltonian graph $G_{k'}$ from Lemma~\ref{lem:gk}
  with $k'\ge k$, $k'>0$ and splice the 
edges of the matching using each of the strong H-plugins ${\cal P}_{M_1}^{str}$,\dots ,${\cal P}_{M_k}^{str}$ at least once.
This removes all vertices of degree $3$ or increases their degree to
$4$.  Furthermore outside the H-plugins only degree $4$ occurs and in the H-plugins exactly all vertex degrees in $M$ occur, while the graph has still one
unique hamiltonian cycle.

To show that $M$ is strongly uhc-realizable for $k\ge 1$, assume a partition $D_1,D_2$ of $M\setminus \{4\}$ to be given.
If $D_2=\emptyset$, to construct the sequence of graphs we can use increasingly large strong H-plugins -- keeping the numbers of vertices of degree $d$ constant for $d\in \{d_1,\dots,d_k\}$.
If $D_2\not=\emptyset$, we can use graphs $G_{k'}$ for increasingly large $k'$ and use the same arbitrarily large number of copies of strong H-plugins ${\cal P}_d^{str}$ for each $d\in D_2$.

\end{proof}

\begin{figure}[tb]
	\centering
        \resizebox{0.65\textwidth}{0.65\textwidth}
        {
          \input{7_seed_chosen_colour.tikz}
          }
	\caption{A $\{4,7\}$-seed $G$ with the unique hamiltonian path $1,2,3 ,\dots ,18$ from $s=1$ to $t=18$. The set $V'$ that fulfills the requirements of
          the definition is the set containing all vertices of degree $3$, except for vertex $t=18$, so $V'=\{5,9,10,13,17\}$.  The uniqueness of the
          hamiltonian path as well as the fact that $G$ -- so also the result of $W_{V'}(G)$ used for splicing in a 3-connected graph -- is $3$-connected, have
          been checked by computer, but as the graph is relatively small, these properties can -- though tedious -- still be checked by hand.}
	\label{fig:7seed}
\end{figure}

\begin{figure}[tb]
	\centering
	\includegraphics[width=0.37\textwidth]{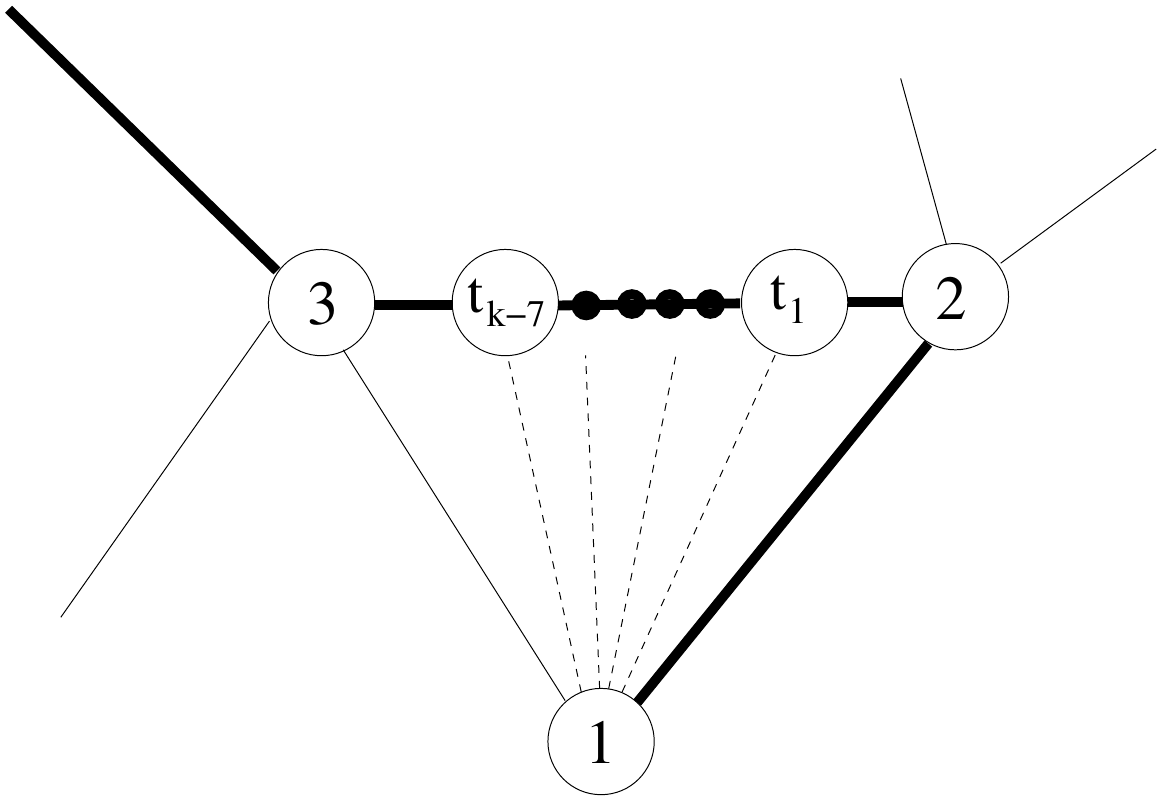}
	\caption{Constructing $\{4,k\}$-seeds for $k>8$.}
	\label{fig:largerk}
\end{figure}

\begin{remark}\label{rem:seeds}

  For each $k\ge 7$ there are $\{4,k\}$-seeds.

  For each $k\ge 8$ there are infinitely many $\{4,k\}$-seeds, so that the corresponding plugins after being used for splicing an edge
  with both endpoints of degree $3$ contain two vertices of degree $k$ and
   all other vertices have degree $4$.

\end{remark}

\begin{proof}
  We start from the $\{4,7\}$-seed $S_7$ in Figure~\ref{fig:7seed}.
  We use the triangle $1,2,3$ for constructing $\{4,k\}$-seeds for $k\ge 8$ as in Figure~\ref{fig:largerk}: new vertices $t_1,\dots ,t_{k-7}$ are
  inserted, the edge $\{2,3\}$ is replaced by the path $2,t_{1},\dots ,t_{k-7},3$, and edges $\{1,t_1\},\dots ,\{1,t_k\}$ are added. Each hamiltonian path from
  $s=1$ to $t=18$ that is not $1,2,t_{1},\dots ,t_{k-7},3,\dots ,18$ could be transformed to a hamiltonian path contradicting the uniqueness of the hamiltonian path in $S_7$.
  Also the connectivity requirements can be easily checked.

  For $k\ge 8$ there is a vertex $t_{k-7}$ and the number of vertices of degree $4$ can be increased by steps of $1$ always producing new $\{4,k\}$-seeds for
  the same $k$. This procedure is described in Figure~\ref{fig:extend_weak}. Any hamiltonian path from $1$ to $18$ traversing the vertices in a different way than given in
  Figure~\ref{fig:extend_weak} would imply a second hamiltonian path from $1$ to $18$ in $S_7$

  \end{proof}

The construction of the $\{4,k\}$-seeds is exclusively to show that such seeds do exist and by no means meant to construct minimal ones. For $k>7$ smaller
$\{4,k\}$-seeds are known -- e.g.\ a $\{4,10\}$-seed with 10 vertices. This $\{4,10\}$-seed has only vertices of degree $2$, $3$, and $4$ and the hamiltonian path
goes from a vertex of degree $2$ to a vertex of degree $3$ -- see Figure~\ref{fig:min10seed}.

\begin{figure}[tb]
	\centering
	\includegraphics[width=0.37\textwidth]{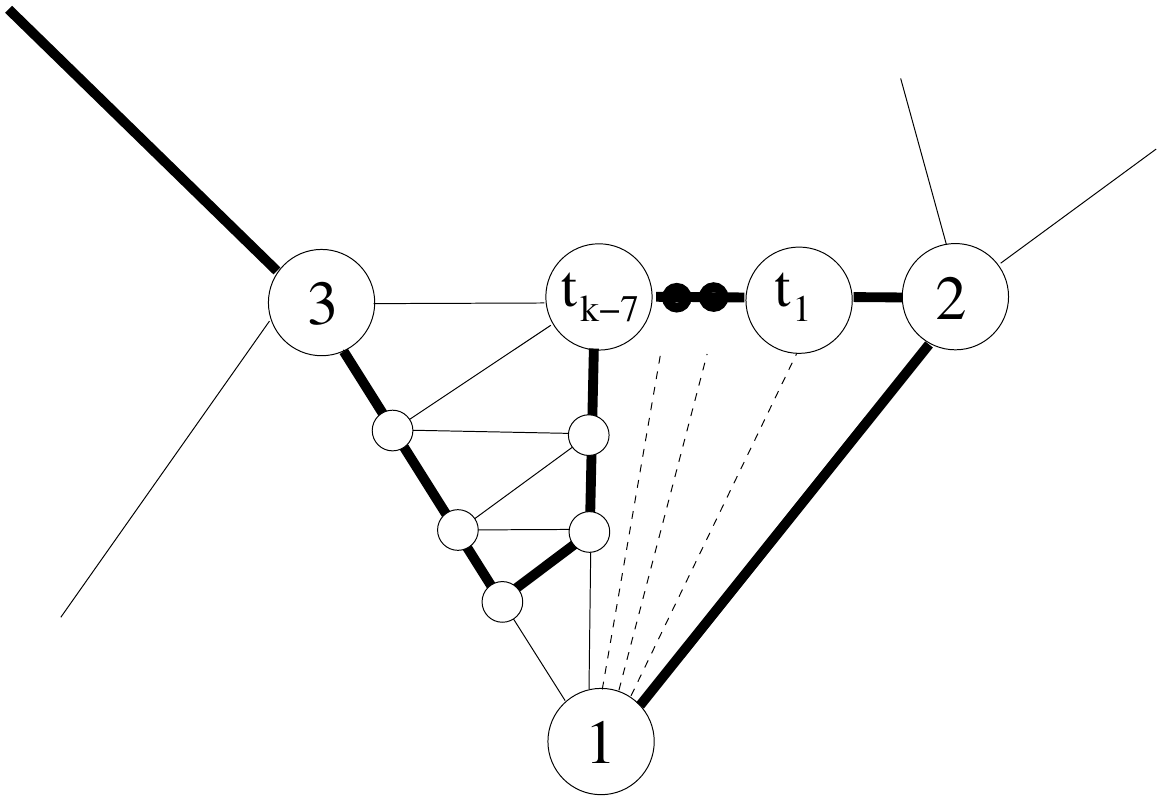}
	\caption{Extending a $\{4,k\}$-seed for $k\ge 8$ constructed from the $\{4,7\}$-seed in Figure~\ref{fig:7seed} by increasing the number of vertices of degree $4$. The number of vertices of degree $4$ can be increased by steps of $1$ vertex.}
	\label{fig:extend_weak}
\end{figure}

Unfortunately in spite of an extensive computer search, no $\{4\}$-, $\{4,5\}$-, or $\{4,6\}$-seeds were found. See Section~\ref{sec:comp} for details.

\begin{figure}[tb]
	\centering
        \resizebox{0.65\textwidth}{0.65\textwidth}
        {
          \input{altseed_6_7_16.tikz}
          }
	\caption{A $\{4,6,7\}$-seed $G$ with the unique hamiltonian path $1,2,3 ,\dots ,16$ from $s=1$ to $t=16$. The set $V'$ that fulfills the requirements of the definition
          is the set containing all vertices of degree $3$.}
	\label{fig:67seed}
\end{figure}

\begin{remark}\label{rem:seeds46k}

  For each $k\ge 7$ there are $\{4,6,k\}$-seeds.

  For each $k\ge 8$ there are infinitely many $\{4,6,k\}$-seeds, so that the corresponding plugins after being used for splicing  an edge with both endpoints of degree $3$ contain one vertex of degree $6$, $2$ vertices of degree $k$, and
  all other vertices have degree $4$.

\end{remark}

\begin{proof}

  In Figure~\ref{fig:67seed} a $\{4,6,7\}$-seed is given that contains a triangle $1,2,3$ and the unique hamiltonian cycle from $1$ to $16$ contains the edges $\{1,2\}$ and $\{2,3\}$.
  Except for vertex $1$ none of the vertices has degree $7$ after splicing an edge, so the seed can be extended in the same way as in the proof of Remark~\ref{rem:seeds} to seeds for larger $k$ and for $k\ge 8$ also to the infinite sequence.

  \end{proof}

\begin{figure}[tb]
	\centering
        \resizebox{0.65\textwidth}{0.65\textwidth}
        {
          \input{altseed_5_6_16_chosen.tikz}
          }
	\caption{A $\{4,5,6\}$-seed $G$ with the unique hamiltonian path $1,2,3 ,\dots ,16$ from $s=1$ to $t=16$. The set $V'$ that fulfills the requirements of the definition
          is the set containing all vertices of degree $3$.}
	\label{fig:56seed}
\end{figure}

\begin{remark}\label{rem:seeds45k}

  For each $k\ge 6$ there are $\{4,5,k\}$-seeds.

  For each $k\ge 7$ there are infinitely many $\{4,5,k\}$-seeds, so that the corresponding plugins after being used for splicing an edge with both endpoints of degree $3$ contain one vertex of degree $5$, $2$ vertices of degree $k$, and
  all other vertices have degree $4$.

\end{remark}

\begin{proof}

  In Figure~\ref{fig:56seed} a $\{4,5,6\}$-seed is given that contains a triangle $1,2,3$ and the unique hamiltonian cycle from $1$ to $16$ contains the edges $\{1,2\}$ and $\{2,3\}$.
  Except for vertex $1$ none of the vertices has degree $6$ after splicing an edge, so the seed can be extended in the same way as in the proof of Remark~\ref{rem:seeds} to seeds for larger $k$ and for $k\ge 7$ also to the infinite sequence.

\end{proof}

Theorem~\ref{thm:set} and Remarks~\ref{rem:seeds}, \ref{rem:seeds46k}, and \ref{rem:seeds45k} now immediately imply the main result for minimum degree $4$:

\begin{theorem}\label{thm:main2}

  \begin{itemize}
\item Except for maybe $\{4\}, \{4,5\}$, and $\{4,6\}$, any set $M=\{4,d_1,d_2,\dots ,d_k\}$ with $4\le d_1 < d_2 < \dots < d_k$ is uhc-realizable.

\item Any set $M=\{4,d_1,d_2,\dots ,d_k\}$ with $8\le d_1 < d_2 < \dots < d_k$ and $k\ge 1$ is strongly uhc-realizable.
  \end{itemize}
  
\end{theorem}

\bigskip

Due to
Theorem~\ref{thm:set} the existence of a $4$-seed implies the existence of a $3$-connected uniquely hamiltonian $4$-regular graph, but in fact also the other
direction is correct:

\begin{figure}[tb]
	\centering
	\includegraphics[width=0.6\textwidth]{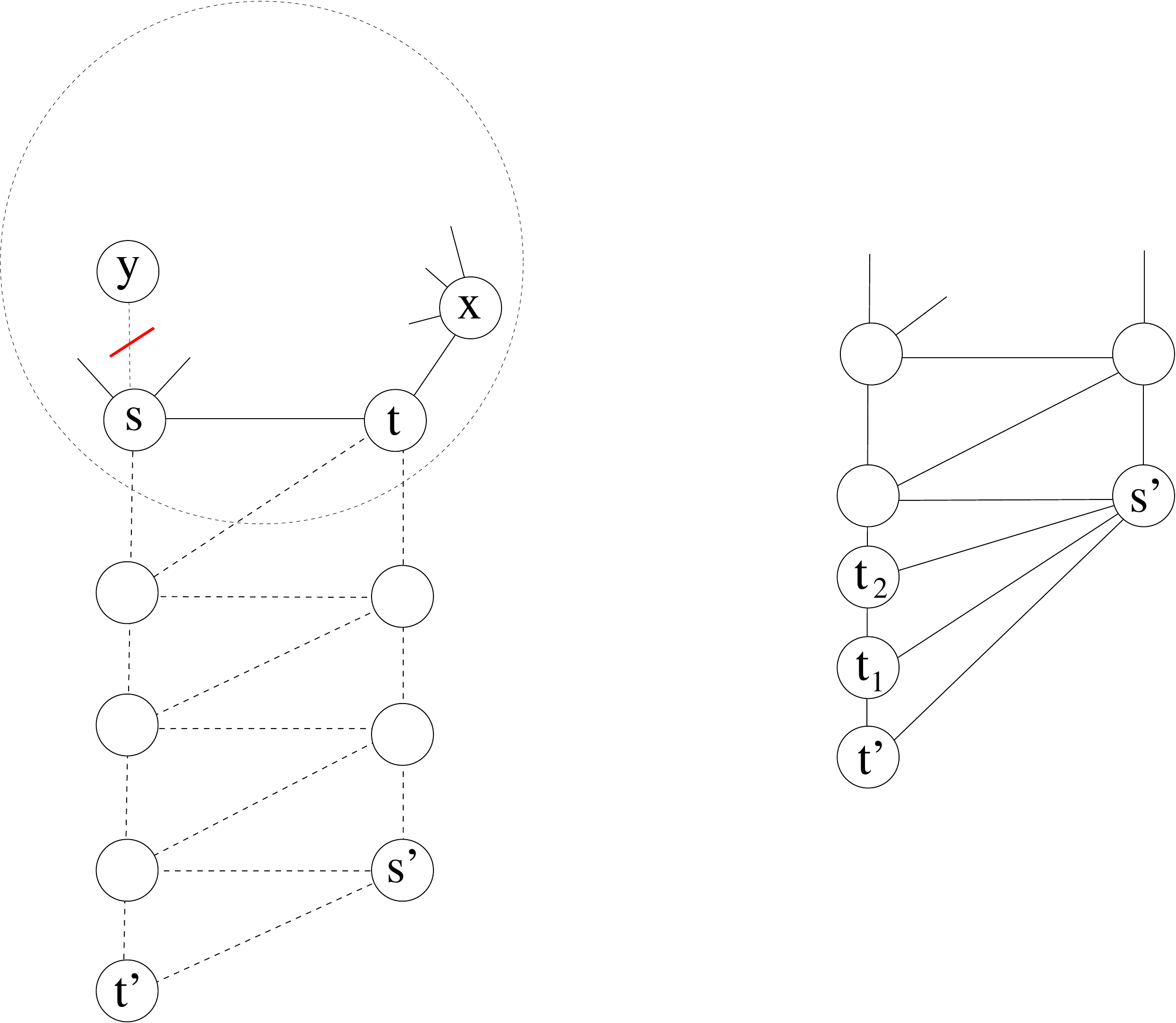}
	\caption{Extending a $\{4\}$-seed with more vertices of degree $4$ and making  $\{4,k\}$-seeds from it, depicted for the example $k=6$, where the set $V'$ from the definition of seeds would be
          $\{y,t',t_1,t_2\}$.}
	\label{fig:extend_4}
\end{figure}

\begin{corollary}\label{cor:notsheehan}
  There is a $3$-connected uniquely hamiltonian $4$-regular graph, if and only if there is a $\{4\}$-seed. In that case there are infinitely many $3$-connected uniquely hamiltonian
  $4$-regular graphs and every set $M$ of natural numbers $d\ge 2$ with $4\in M$ and $|M|\ge 2$ is strongly uhc-realizable.
\end{corollary}

\begin{proof}
  From a $3$-connected uniquely hamiltonian $4$-regular graph $G$ we can get a $\{4\}$-seed with $\deg(s)=3$ and  $\deg(t)=2$ 
  by choosing a vertex of $G$ as $s$, subdivide an edge $\{s,x\}$ on the hamiltonian cycle incident with $s$ with a new vertex $t$, and remove an edge $\{s,y\}$ that is not on the hamiltonian cycle.
  The set $\{y,t\}$ shows that it is a $\{4\}$-seed. The $3$-connectivity after using plugins constructed from it follows by standard arguments from Menger's theorem. A way to construct
  seeds with more vertices of degree $4$ and $\{4,k\}$-seeds for $k>4$ is given in Figure~\ref{fig:extend_4}.
  The rest of the statement is a direct consequence of Remark~\ref{rem:2}, Theorem~\ref{thm:alldegcubic}, and Theorem~\ref{thm:set}. 
\end{proof}

Furthermore, for $4$-regular graphs, the existence of a $2$-connected uniquely hamiltonian graph also implies the existence of a $3$-connected uniquely hamiltonian graph:

\begin{lemma}

  There is a $3$-connected uniquely hamiltonian $4$-regular graph, if and only if there is a $2$-connected uniquely hamiltonian $4$-regular graph.

\end{lemma}

\begin{proof}

  As $3$-connected graphs are also $2$-connected, the only thing to prove is that the existence of a uniquely hamiltonian $4$-regular graph with a $2$-cut implies the existence
  of a $3$-connected uniquely hamiltonian $4$-regular graph.

  Let $G=(V,E)$ be a uniquely hamiltonian $4$-regular graph with a $2$-cut and $\{s,t\}$ be vertices of a $2$-cut, so that one of the components of $G[V\setminus \{s,t\}]$ -- say $C_0$ -- has minimum size.
  Let $G_0=G[C_0\cup\{s,t\}]$. Then there is a unique hamiltonian path in $G_0$ from $s$ to $t$ and due to the minimality of $C_0$ the vertices $s$ and $t$ have degree at least $2$ in $G_0$. If
  one has degree $2$, they are non-adjacent. As the number of vertices with odd degree must be even and as they both have neighbours in more than one component, they both have degree
  $2$ or both have degree $3$. In case of degree $2$ we can add the edge $\{s,t\}$, so that in each case we have a graph, which we will call again $G_0$ with a unique hamiltonian path $P_H$ from $s$ to $t$,
  where $s$ and $t$ are of degree $3$ and all other vertices of degree $4$. Let now $G_0^v$ be $G_0$ with an edge $e\not=\{s,t\}$ that is not part of $P_H$ subdivided with a new vertex $v$. By construction $G_0^v$
  does not have a hamiltonian path from $s$ to $t$, but a unique hamiltonian path in $(G_0^v)_{-v}=G_0$. So $G_0^v$ is a strong H-plugin that when applied to two connected copies of $P^-$ like in Lemma~\ref{lem:gk}
  gives a $4$-regular uniquely hamiltonian graph.

  It remains to be shown that for a $3$-connected graph $G'$ and suitable $x,y\in G'$ the graph $O(x,y,G_0^v)$ is $3$-connected. It is sufficient to show that
  the graph $G_1$ obtained from $G_0^v$ by adding a new vertex $v'$ and connecting it to $s,t,$ and $v$ is $3$-connected.

  Assume to the contrary that $G_1$ has a $2$-cut $K$. Note that $K\not= \{s,t\}$ as $C_0$ is a component and $v$ and through $v$ also $v'$ are connected to
  it. If $s$ and $t$ are in different components of $G_1\setminus K$, then the common neighbour $v'$ must be in $K$.  So $K\setminus \{v'\}$ is a $1$-cut of
  $G_0^v$. If $K=\{v,v'\}$, choose $w$ as a neighbour of $v$ different from $s,t$, otherwise let $w$ be the vertex in $K\setminus \{v'\}$. Then $w$ is a
  cutvertex of $G_0^v$ and also of $G_0$. Together with $s$ or $t$ it forms a $2$-cut contradicting the minimality of $C_0$.

  If $s$ and $t$ are in the same component of $G_1\setminus K$ or one is in $K$, there is a vertex $x\not\in \{v,v'\}$ in a component not containing $s$ or $t$. But then $K$ -- possibly after
  replacing $v$ or $v'$ in $K$ by a neighbour -- again contradicts the minimality of $C_0$, so $G_1$ does not have a $2$-cut.

\end{proof}

In \cite{UHC4regmult} Fleischner proved that there are $4$-regular uniquely hamiltonian multigraphs and in fact $2k$-regular uniquely hamiltonian multigraphs with arbitrarily high degree.
Another direct consequence of Lemma~\ref{lem:gk} is the following simple generalisation:

\begin{corollary}

  For a set $M=\{d_1,\dots,d_k\}$ with $2\le d_1<d_2<\dots <d_k$ of natural numbers there is a uniquely hamiltonian  multigraph $G$ with $M_{deg}(G)=M$ if and only if $M$ contains an even number.
  In that case there are infinitely many $3$-connected uniquely hamiltonian multigraphs $G$ with $M_{deg}(G)=M$. 

\end{corollary}

\begin{proof}

   In \cite{Thomason_odd} it is shown that uniquely hamiltonian multigraphs do not exist if all degrees are odd, so we only have to prove that they do exist if an even degree is contained.

   For $2\in M$ this is even proven for simple graphs in Remark~\ref{rem:2}, so assume that all elements of $M$ are at least $3$. Taking graphs $G_{k'}$ with $k'\ge k$ from Lemma~\ref{lem:gk}
   with the matching and $2$-factor with the described properties,
  we can multiply the edges of the $2$-factor containing the $4$-regular vertices until the vertices all have an even degree contained in $M$. For each remaining degree $d_i$, we can now choose
  an edge in the matching and multiply it until it has degree $d_i$. If there are still vertices of degree $3$ left and $3\not\in M$, we can multiply the corresponding edges
  of the matching until a degree in $M$ is reached.

  \end{proof}

\section{ Computational results }\label{sec:comp}


All seeds displayed in this article (and many more) were found by computer. Generating and testing all graphs with certain degrees would be too time consuming,
so two specialized programs were developed, one of them mainly to test the second one that was used for the most time consuming runs.
The programs were designed to construct seeds where the degree sets with limits for the numbers of vertices with
each degree as well as the degrees of $s$ and $t$ are given as a parameter.  The programs start with a hamiltonian path $s=1,2,\dots ,n-1,t=n$ and then add edges in a
way that given degree restrictions are respected and that no second hamiltonian path from $1$ to $n$ is introduced.

The smallest $\{4,10\}$-seed has only $10$ vertices, $\{\deg(s),\deg(t)\}=\{2,3\}$, and 
the maximum degree is $4$ -- see Figure~\ref{fig:min10seed}.
Though  $\{\deg(s),\deg(t)\}=\{2,3\}$ and maximum degree $4$
in the seed seems a good choice as it implies the smallest possible number of edges in a seed with a given number of vertices, for $k<10$ no $\{4,k\}$-seeds with this structure exist up to
$|V|\le 21$. In fact for odd $k$, no such seeds can exist as they would need an odd number of vertices of the only odd degree, which is $3$.

The smallest $\{4,9\}$-seeds have $14$ vertices and $\{\deg(s),\deg(t)\}=\{3,7\}$ or $\{\deg(s),\deg(t)\}=\{3,8\}$ and
the smallest $\{4,8\}$-seeds have $14$ vertices and $\{\deg(s),\deg(t)\}=\{3,7\}$ or $\{\deg(s),\deg(t)\}=\{7,7\}$. Except for $s,t$, also for these seeds the maximum degree is $4$.
In fact one of the $\{4,8\}$-seeds is also a $\{4,9\}$-seed. It is given in Figure~\ref{fig:8and9seed}.

\begin{figure}[tb]
	\centering
        \resizebox{0.43\textwidth}{0.43\textwidth}
        {
          \input{minseed10.tikz}
          }
	\caption{A graph with a unique hamiltonian path from $s=1$ to $t=10$. The set $V'=\{3,4,5,\dots ,10\}$ shows that it is a $\{4,10\}$-seed. It is the unique smallest  $\{4,10\}$-seed. }
	\label{fig:min10seed}
\end{figure}

\begin{figure}[tb]
	\centering
        \resizebox{0.55\textwidth}{0.55\textwidth}
        {
          \input{8and9seed_14.tikz}
          }
	\caption{A graph with a unique hamiltonian path from $s=1$ to $t=14$. The set $V'=\{4,5,6,8,9,12\}$ shows that it is a $\{4,8\}$-seed and the set $V'=\{1,4,5,6,8,9,12\}$ shows that it is a $\{4,9\}$-seed. There are no smaller  $\{4,8\}$- or $\{4,9\}$-seeds. }
	\label{fig:8and9seed}
\end{figure}

Unfortunately the computation of seeds is very time consuming. Testing all possible sets of degrees of $\{4,7\}$-seeds on $15$ vertices already took about $100$ days of CPU time
on an AMD EPYC 7552 running with 2.2 to 3.3 GHz with full load.
The possible presence of vertices with degree larger than $4$ {\em inside} the seed -- that is: at a vertex different from $s$ and $t$ -- has a large impact on the time consumption.
Not allowing vertices with degree larger than $4$ inside the seed, the search for $\{4,7\}$-seeds on $15$ vertices needed about $37$ minutes on the same machine.
As for the smallest $\{4,k\}$-seeds for $k\in \{8,9,10\}$, no such vertices were present, for $k\le 7$ we focused on seeds without vertices with degree larger than $4$ inside.

To be exact: for $k\in\{6,7\}$ we did a complete search only up to $15$ vertices. No $\{4,6\}$- or $\{4,7\}$-seeds exist for these vertex numbers.
For larger vertex numbers we focused on seeds without interior vertices with large degree. The smallest such $7$-seeds have $18$ vertices -- an example is given in Figure~\ref{fig:7seed}.
For $6$-seeds the existence of seeds without internal vertices of large degree was only checked up to $17$ vertices. No such $6$-seeds were found. For $18,19$ and $20$ vertices we restricted the search
to $3$ cases: for  no internal vertices with degree larger than $4$ the cases $\{\deg(s),\deg(t)\}=\{5,5\}$ and $\{\deg(s),\deg(t)\}=\{3,5\}$ were checked.
For one internal vertex with degree $5$, the case $\{\deg(s),\deg(t)\}=\{3,3\}$ with $5$ vertices of degree $3$, one (internal)
vertex of degree $5$ and the rest of degree $4$ was checked. No such seeds were found and the total CPU time needed was about $18$ years on an AMD EPYC 7532.

Even for carefully designed and implemented algorithms independent tests are necessary.
As runs without any output are not very good tests for the programs, the two programs were also compared when generating $10$-seeds and $12$-seeds with $\{\deg(s),\deg(t)\}=\{2,3\}$ and no internal vertices with degree larger than $4$.
The output of the two programs was compared for their number and for isomorphism up to $20$ vertices.  For $10$-seeds there were in total $4.689$ non-isomorphic seeds and
for $12$-seeds there were in total $1.414.640$ non-isomorphic seeds. 
In addition $k$-seeds with $\{\deg(s),\deg(t)\}=\{3,k-1\}$, no internal vertices of degree larger than $4$, and $6\le k\le 10$ on up to $16$ vertices (in total $4.907$ seeds) were compared.
For seeds, {\em isomorphism} means that the two endpoints of the hamiltonian path are
marked vertices and are distinguished from the other vertices, so some seeds that are non-isomorphic as seeds can be isomorphic as graphs. There was complete agreement.
The program used for the larger runs can be obtained from the authors.

\section{ Final remarks}

\begin{figure}[tb]
	\centering
	\includegraphics[width=0.6\textwidth]{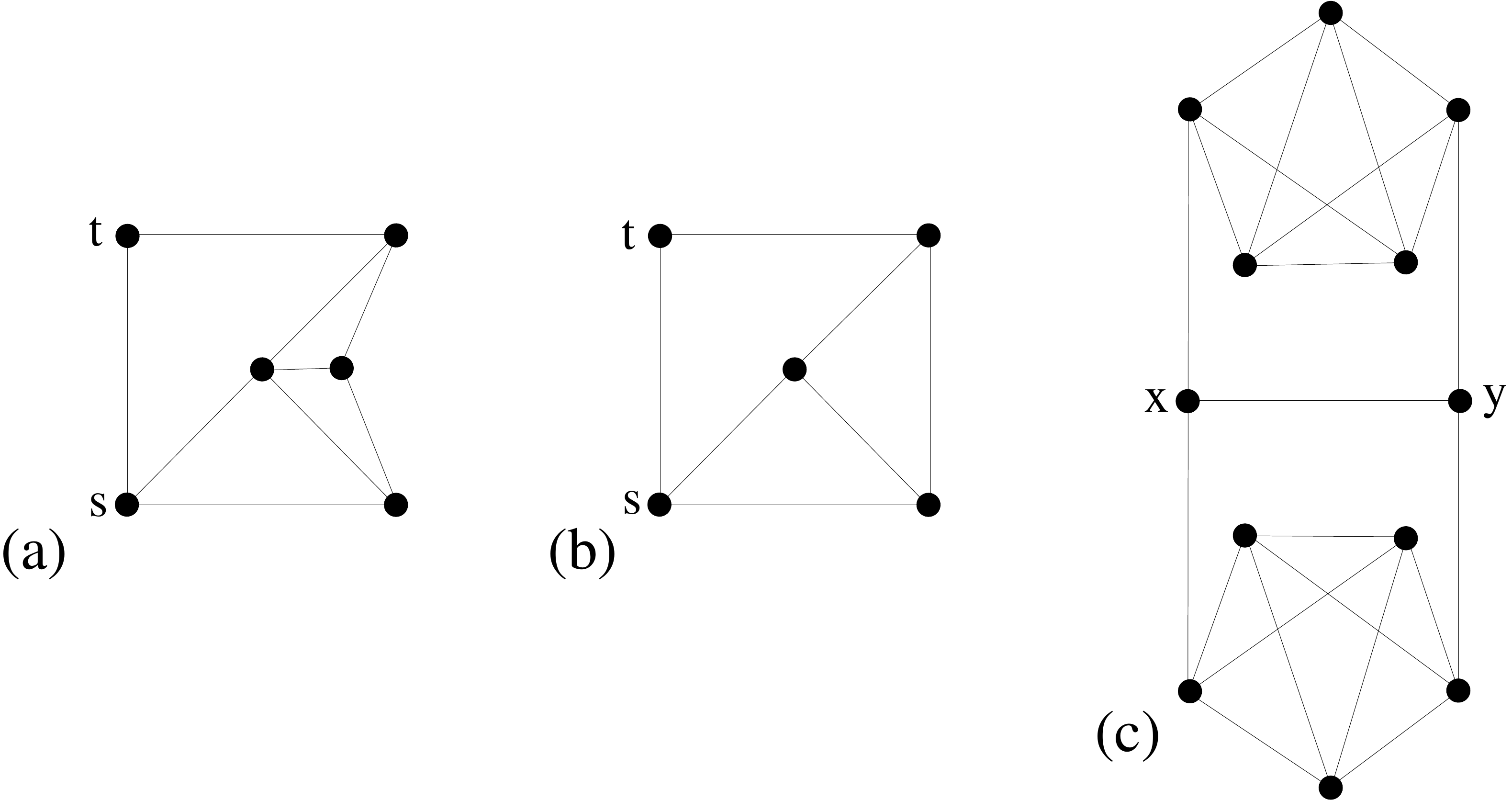}
	\caption{The splicing operation for more than one hamiltonian cycle with a {\em generalized 4-seed} and a {\em generalized 6-seed}.}
	\label{fig:count}
\end{figure}

In this article we are interested only in uniquely hamiltonian graphs. Nevertheless the method of splicing can also be useful when constructing graphs with few
hamiltonian cycles. We will only give a short sketch of the possibilities. We will not formally state results, as we do not give formal proofs. The following statements should
be considered as preliminary as long as no proofs are given somewhere.

If we allow $n_s$ hamiltonian paths from $s$ to $t$ in a seed and $n_G$ hamiltonian cycles in a graph $G$ -- none of them containing the edge $e$ of $G$ -- then
with the otherwise same prerequisites of Lemma~\ref{lem:op}, the proof can be repeated, this time showing that the result after splicing has $n_s\cdot n_G$
hamiltonian cycles. This implies that for any set $M$ of natural numbers with minimum $4$ there is a constant $C$ and an infinite series of graphs
with degree set $M$ and at most $C$ hamiltonian cycles. In fact there is also one constant working as an upper bound for all sets $M$.
The constants we get from our proof that used $P^-$ are nevertheless very large and far worse for the 4-regular case than in
\cite{fewlongcyc}. For better constants one has to search for starting graphs that need fewer splicing operations, but can have more than one hamiltonian
cycle. An example is the construction in \cite{fewlongcyc} proving that there are infinitely many (2-connected) $4$-regular graphs with $144$ hamiltonian cycles. It was found
and proven in a completely different way, but can be interpreted making use of splicing:

The graph in Figure~\ref{fig:count}(c) has $36$ hamiltonian cycles -- none of them containing $\{x,y\}$. Furthermore removing $y$, the graph is
non-hamiltonian. The {\em generalized 4-seed} (that is: allowing more than one hamiltonian path from $s$ to $t$) in Figure~\ref{fig:count}(a) has $4$ hamiltonian
paths from $s$ to $t$, so with plugins obtained from it and its extensions, the results of splicing $\{x,y\}$ have $144$ hamiltonian cycles.  The 
  generalized seed in Figure~\ref{fig:count}(b) has $2$ hamiltonian paths from $s$ to $t$ and would give one vertex of degree $6$, so splicing $\{x,y\}$ would
give $72$ hamiltonian cycles for the degree set $M=\{4,6\}$ and replacing a vertex of degree $3$ by a triangle also for $M=\{4,8\}$.

\bigskip

All graphs explicitly given in the previous sections can be inspected at and downloaded from the database {\em House of Graphs} \cite{HoG2}. They can be found by
searching for the keyword \verb+UHG_degree_sequence+.

All properties about small graphs stated here have been checked by computer, but can -- though sometimes tedious -- be confirmed by hand.

\bigskip

\section{ Acknowledgement}

A large part of the computational resources (Stevin Supercomputer Infrastructure) and services used in this work were provided by the VSC (Flemish Supercomputer Center), funded by Ghent University, FWO and the Flemish Government – department EWI.

\bibliographystyle{plain}

\end{document}

%% file: one_of_smallest_uhc_3_4_colour.tikz
\begin{tikzpicture}[scale=0.06]
\def\vertexscale{1.20}
\def\labelscale{1.50}
\node [circle,blue,draw,scale=\vertexscale] (1) at (-63.88772,-6.43188) {1};
\node [circle,blue,draw,scale=\vertexscale] (2) at (-17.35741,-49.97904) {2};
\node [circle,black,draw,scale=\vertexscale] (3) at (-28.27114,66.22080) {3};
\node [circle,blue,draw,scale=\vertexscale] (4) at (69.73301,-25.96196) {4};
\node [circle,blue,draw,scale=\vertexscale] (5) at (58.36870,-49.64097) {5};
\node [circle,blue,draw,scale=\vertexscale] (6) at (39.50924,14.71820) {6};
\node [circle,blue,draw,scale=\vertexscale] (7) at (10.30821,-23.76097) {7};
\node [circle,blue,draw,scale=\vertexscale] (8) at (25.03614,-39.60628) {8};
\node [circle,blue,draw,scale=\vertexscale] (9) at (52.01586,-18.81717) {9};
\node [circle,blue,draw,scale=\vertexscale] (10) at (63.88772,6.43188) {10};
\node [circle,blue,draw,scale=\vertexscale] (11) at (17.35741,49.97904) {11};
\node [circle,black,draw,scale=\vertexscale] (12) at (28.27114,-66.22080) {12};
\node [circle,blue,draw,scale=\vertexscale] (13) at (-69.73301,25.96196) {13};
\node [circle,blue,draw,scale=\vertexscale] (14) at (-58.36870,49.64097) {14};
\node [circle,blue,draw,scale=\vertexscale] (15) at (-39.50924,-14.71820) {15};
\node [circle,blue,draw,scale=\vertexscale] (16) at (-10.30821,23.76097) {16};
\node [circle,blue,draw,scale=\vertexscale] (17) at (-25.03614,39.60628) {17};
\node [circle,blue,draw,scale=\vertexscale] (18) at (-52.01586,18.81717) {18};
\tkzDefPoint(-79.33533,-60.87614){19}
\tkzDefPoint(-38.26834,-92.38795){20}
\tkzDefPoint(-60.87614,-79.33533){21}
\tkzDefPoint(-13.05262,-99.14449){22}
\tkzDefPoint(13.05262,-99.14449){23}
\tkzDefPoint(99.14449,-13.05262){24}
\tkzDefPoint(99.14449,13.05262){25}
\tkzDefPoint(60.87614,-79.33533){26}
\tkzDefPoint(79.33533,-60.87614){27}
\tkzDefPoint(-79.33533,60.87614){28}
\tkzDefPoint(-60.87614,79.33533){29}
\tkzDefPoint(-99.14449,-13.05262){30}
\tkzDefPoint(-99.14449,13.05262){31}
\tkzDefPoint(60.87614,79.33533){32}
\tkzDefPoint(79.33533,60.87614){33}
\tkzDefPoint(-13.05262,99.14449){34}
\tkzDefPoint(13.05262,99.14449){35}
\tkzDefPoint(-92.38795,-38.26834){36}
\tkzDefPoint(-92.38795,38.26834){37}
\tkzDefPoint(38.26834,-92.38795){38}
\tkzDefPoint(38.26834,92.38795){39}
\tkzDefPoint(92.38795,38.26834){40}
\tkzDefPoint(92.38795,-38.26834){41}
\tkzDefPoint(-38.26834,92.38795){42}
\draw [black] (1) to (30);
\node [draw=none,fill=none,scale=\labelscale] () at (-104.10171,-13.70525) {2};
\draw [black] (1) to (18);
\draw [black] (1) to (15);
\draw [black] (2) to (22);
\node [draw=none,fill=none,scale=\labelscale] () at (-13.70525,-104.10171) {1};
\draw [black] (2) to (21);
\node [draw=none,fill=none,scale=\labelscale] () at (-63.91995,-83.30210) {3};
\draw [black] (2) to (7);
\draw [black] (3) to (29);
\node [draw=none,fill=none,scale=\labelscale] () at (-63.91995,83.30210) {2};
\draw [black] (3) to (34);
\node [draw=none,fill=none,scale=\labelscale] () at (-13.70525,104.10171) {4};
\draw [black] (3) to (17);
\draw [black] (3) to (14);
\draw [black] (4) to (5);
\draw [black] (4) to (9);
\draw [black] (4) to (24);
\node [draw=none,fill=none,scale=\labelscale] () at (104.10171,-13.70525) {3};
\draw [black] (5) to (27);
\node [draw=none,fill=none,scale=\labelscale] () at (83.30210,-63.91995) {6};
\draw [black] (5) to (12);
\draw [black] (6) to (7);
\draw [black] (6) to (33);
\node [draw=none,fill=none,scale=\labelscale] () at (83.30210,63.91995) {5};
\draw [black] (6) to (10);
\draw [black] (7) to (8);
\draw [black] (8) to (12);
\draw [black] (8) to (9);
\draw [black] (9) to (10);
\draw [black] (10) to (25);
\node [draw=none,fill=none,scale=\labelscale] () at (104.10171,13.70525) {11};
\draw [black] (11) to (35);
\node [draw=none,fill=none,scale=\labelscale] () at (13.70525,104.10171) {10};
\draw [black] (11) to (32);
\node [draw=none,fill=none,scale=\labelscale] () at (63.91995,83.30210) {12};
\draw [black] (11) to (16);
\draw [black] (12) to (23);
\node [draw=none,fill=none,scale=\labelscale] () at (13.70525,-104.10171) {13};
\draw [black] (12) to (26);
\node [draw=none,fill=none,scale=\labelscale] () at (63.91995,-83.30210) {11};
\draw [black] (13) to (14);
\draw [black] (13) to (18);
\draw [black] (13) to (31);
\node [draw=none,fill=none,scale=\labelscale] () at (-104.10171,13.70525) {12};
\draw [black] (14) to (28);
\node [draw=none,fill=none,scale=\labelscale] () at (-83.30210,63.91995) {15};
\draw [black] (15) to (16);
\draw [black] (15) to (19);
\node [draw=none,fill=none,scale=\labelscale] () at (-83.30210,-63.91995) {14};
\draw [black] (16) to (17);
\draw [black] (17) to (18);
\tkzDefPoint(-93.49426,-35.47990){A}
\tkzDefPoint(-93.49426,35.47990){B}
\tkzDefPoint(0.0,0.0){C}
\tkzDrawArc[<-,line width=0.9mm, red](C,B)(A)
\tkzDefPoint(-91.19850,41.02235){A}
\tkzDefPoint(-41.02235,91.19850){B}
\tkzDefPoint(0.0,0.0){C}
\tkzDrawArc[<-,line width=0.9mm, blue](C,B)(A)
\tkzDefPoint(-35.47990,93.49426){A}
\tkzDefPoint(35.47990,93.49426){B}
\tkzDefPoint(0.0,0.0){C}
\tkzDrawArc[<-,line width=0.9mm, green](C,B)(A)
\tkzDefPoint(41.02235,91.19850){A}
\tkzDefPoint(91.19850,41.02235){B}
\tkzDefPoint(0.0,0.0){C}
\tkzDrawArc[<-,line width=0.9mm, orange](C,B)(A)
\tkzDefPoint(93.49426,35.47990){A}
\tkzDefPoint(93.49426,-35.47990){B}
\tkzDefPoint(0.0,0.0){C}
\tkzDrawArc[->,line width=0.9mm, green](C,B)(A)
\tkzDefPoint(91.19850,-41.02235){A}
\tkzDefPoint(41.02235,-91.19850){B}
\tkzDefPoint(0.0,0.0){C}
\tkzDrawArc[->,line width=0.9mm, orange](C,B)(A)
\tkzDefPoint(35.47990,-93.49426){A}
\tkzDefPoint(-35.47990,-93.49426){B}
\tkzDefPoint(0.0,0.0){C}
\tkzDrawArc[->,line width=0.9mm, red](C,B)(A)
\tkzDefPoint(-41.02235,-91.19850){A}
\tkzDefPoint(-91.19850,-41.02235){B}
\tkzDefPoint(0.0,0.0){C}
\tkzDrawArc[->,line width=0.9mm, blue](C,B)(A)
\node [black,circle,draw,fill=white,scale=0.75,line width=1mm] (20) at (-38.26834,-92.38795) {};
\node [black,circle,draw,fill=white,scale=0.75,line width=1mm] (36) at (-92.38795,-38.26834) {};
\node [black,circle,draw,fill=white,scale=0.75,line width=1mm] (37) at (-92.38795,38.26834) {};
\node [black,circle,draw,fill=white,scale=0.75,line width=1mm] (38) at (38.26834,-92.38795) {};
\node [black,circle,draw,fill=white,scale=0.75,line width=1mm] (39) at (38.26834,92.38795) {};
\node [black,circle,draw,fill=white,scale=0.75,line width=1mm] (40) at (92.38795,38.26834) {};
\node [black,circle,draw,fill=white,scale=0.75,line width=1mm] (41) at (92.38795,-38.26834) {};
\node [black,circle,draw,fill=white,scale=0.75,line width=1mm] (42) at (-38.26834,92.38795) {};
\end{tikzpicture}

%% file: 3plus2plugin.tikz
\begin{tikzpicture}[scale=0.065]
\def\vertexscale{1.70}
\def\labelscale{2.40}
\node [circle,black,draw,scale=\vertexscale] (1) at (31.62051,31.62051) {s=1};
\node [circle,blue,draw,scale=\vertexscale] (2) at (57.47854,0.43410) {2};
\node [circle,blue,draw,scale=\vertexscale] (3) at (27.61963,-21.66896) {3};
\node [circle,blue,draw,scale=\vertexscale] (4) at (5.30558,5.30558) {4};
\node [circle,blue,draw,scale=\vertexscale] (5) at (-21.66896,27.61963) {5};
\node [circle,blue,draw,scale=\vertexscale] (6) at (0.43410,57.47854) {6};
\node [circle,blue,draw,scale=\vertexscale] (7) at (1.18852,-58.92225) {7};
\node [circle,blue,draw,scale=\vertexscale] (8) at (-35.29971,-35.29971) {8};
\node [circle,black,draw,scale=\vertexscale] (9) at (-15.16392,-15.16392) {t=9};
\node [circle,blue,draw,scale=\vertexscale] (10) at (-58.92225,1.18852) {v=10};
\tkzDefPoint(-100.00000,-0.00000){11}
\tkzDefPoint(-70.71068,-70.71068){12}
\tkzDefPoint(0.00000,-100.00000){13}
\tkzDefPoint(100.00000,0.00000){14}
\tkzDefPoint(-0.00000,100.00000){15}
\tkzDefPoint(-70.71068,70.71068){16}
\tkzDefPoint(70.71068,70.71068){17}
\tkzDefPoint(70.71068,-70.71068){18}
\draw [black] (1) to (2);
\draw [black] (1) to (6);
\draw [black] (2) to (14);
\node [draw=none,fill=none,scale=\labelscale] () at (111.00000,0.00000) {10};
\draw [black] (2) to (3);
\draw [black] (3) to (7);
\draw [black] (3) to (4);
\draw [black] (4) to (5);
\draw [black] (4) to (9);
\draw [black] (5) to (6);
\draw [black] (5) to (10);
\draw [black] (6) to (15);
\node [draw=none,fill=none,scale=\labelscale] () at (-0.00000,111.00000) {7};
\draw [black] (7) to (8);
\draw [black] (7) to (13);
\node [draw=none,fill=none,scale=\labelscale] () at (0.00000,-111.00000) {6};
\draw [black] (8) to (9);
\draw [black] (8) to (10);
\draw [black] (10) to (11);
\node [draw=none,fill=none,scale=\labelscale] () at (-111.00000,-0.00000) {2};
\tkzDefPoint(-68.55786,72.79986){A}
\tkzDefPoint(68.55786,72.79986){B}
\tkzDefPoint(0.0,0.0){C}
\tkzDrawArc[<-,line width=0.9mm, red](C,B)(A)
\tkzDefPoint(72.79986,68.55786){A}
\tkzDefPoint(72.79986,-68.55786){B}
\tkzDefPoint(0.0,0.0){C}
\tkzDrawArc[<-,line width=0.9mm, blue](C,B)(A)
\tkzDefPoint(68.55786,-72.79986){A}
\tkzDefPoint(-68.55786,-72.79986){B}
\tkzDefPoint(0.0,0.0){C}
\tkzDrawArc[->,line width=0.9mm, red](C,B)(A)
\tkzDefPoint(-72.79986,-68.55786){A}
\tkzDefPoint(-72.79986,68.55786){B}
\tkzDefPoint(0.0,0.0){C}
\tkzDrawArc[->,line width=0.9mm, blue](C,B)(A)
\node [black,circle,draw,fill=white,scale=0.75,line width=1mm] (12) at (-70.71068,-70.71068) {};
\node [black,circle,draw,fill=white,scale=0.75,line width=1mm] (16) at (-70.71068,70.71068) {};
\node [black,circle,draw,fill=white,scale=0.75,line width=1mm] (17) at (70.71068,70.71068) {};
\node [black,circle,draw,fill=white,scale=0.75,line width=1mm] (18) at (70.71068,-70.71068) {};
\end{tikzpicture}

%% file: fleischner_pmin_colour.tikz
\begin{tikzpicture}[scale=0.06]
\def\vertexscale{1.20}
\def\labelscale{1.30}
\node [circle,blue,draw,scale=\vertexscale] (1) at (26.14974,5.81753) {s=1};
\node [circle,blue,draw,scale=\vertexscale] (2) at (21.20151,30.48761) {2};
\node [circle,black,draw,scale=\vertexscale] (3) at (-7.07195,36.47579) {3};
\node [circle,blue,draw,scale=\vertexscale] (4) at (21.75576,61.75533) {4};
\node [circle,blue,draw,scale=\vertexscale] (5) at (26.84851,-56.62280) {5};
\node [circle,blue,draw,scale=\vertexscale] (6) at (56.62280,-26.84851) {6};
\node [circle,blue,draw,scale=\vertexscale] (7) at (-61.75533,-21.75576) {7};
\node [circle,black,draw,scale=\vertexscale] (8) at (-36.47579,7.07195) {8};
\node [circle,blue,draw,scale=\vertexscale] (9) at (-30.48761,-21.20151) {9};
\node [circle,blue,draw,scale=\vertexscale] (10) at (-5.81753,-26.14974) {10};
\node [circle,blue,draw,scale=\vertexscale] (11) at (-9.96026,-56.90049) {11};
\node [circle,blue,draw,scale=\vertexscale] (12) at (-33.12632,65.20777) {12};
\node [circle,black,draw,scale=\vertexscale] (13) at (-35.55971,35.55971) {13};
\node [circle,blue,draw,scale=\vertexscale] (14) at (-65.20777,33.12632) {14};
\node [circle,blue,draw,scale=\vertexscale] (15) at (56.90049,9.96026) {t=15};
\tkzDefPoint(-96.59258,25.88190){16}
\tkzDefPoint(-70.71068,-70.71068){17}
\tkzDefPoint(-96.59258,-25.88190){18}
\tkzDefPoint(-25.88190,-96.59258){19}
\tkzDefPoint(25.88190,-96.59258){20}
\tkzDefPoint(96.59258,25.88190){21}
\tkzDefPoint(96.59258,-25.88190){22}
\tkzDefPoint(-25.88190,96.59258){23}
\tkzDefPoint(25.88190,96.59258){24}
\tkzDefPoint(-70.71068,70.71068){25}
\tkzDefPoint(70.71068,70.71068){26}
\tkzDefPoint(70.71068,-70.71068){27}
\draw [black] (1) to (2);
\draw [black] (1) to (15);
\draw [->, >= triangle 90, thick, black] (10) to (1);
\draw [black] (2) to (3);
\draw [<->, >= triangle 90, thick, black] (2) to (4);
\draw [black] (3) to (8);
\draw [black] (3) to (13);
\draw [black] (3) to (4);
\draw [black] (4) to (24);
\node [draw=none,fill=none,scale=\labelscale] () at (27.17600,101.42221) {5};
\draw [black] (5) to (20);
\node [draw=none,fill=none,scale=\labelscale] () at (27.17600,-101.42221) {4};
\draw [black] (5) to (11);
\draw [black] (5) to (6);
\draw [black] (6) to (22);
\node [draw=none,fill=none,scale=\labelscale] () at (101.42221,-27.17600) {7};
\draw [->, >= triangle 90, thick, black] (6) to (15);
\draw [<->, >= triangle 90, thick, black] (9) to (7);
\draw [black] (7) to (18);
\node [draw=none,fill=none,scale=\labelscale] () at (-101.42221,-27.17600) {6};
\draw [black] (7) to (8);
\draw [black] (8) to (13);
\draw [black] (8) to (9);
\draw [black] (9) to (10);
\draw [black] (10) to (11);
\draw [black] (11) to (19);
\node [draw=none,fill=none,scale=\labelscale] () at (-27.17600,-101.42221) {12};
\draw [black] (12) to (23);
\node [draw=none,fill=none,scale=\labelscale] () at (-27.17600,101.42221) {11};
\draw [black] (12) to (13);
\draw [<->, >= triangle 90, thick, black] (12) to (14);
\draw [black] (13) to (14);
\draw [black] (14) to (16);
\node [draw=none,fill=none,scale=\labelscale] () at (-101.42221,27.17600) {15};
\draw [black] (15) to (21);
\node [draw=none,fill=none,scale=\labelscale] () at (101.42221,27.17600) {14};
\tkzDefPoint(-68.55786,72.79986){A}
\tkzDefPoint(68.55786,72.79986){B}
\tkzDefPoint(0.0,0.0){C}
\tkzDrawArc[<-,line width=0.9mm, red](C,B)(A)
\tkzDefPoint(72.79986,68.55786){A}
\tkzDefPoint(72.79986,-68.55786){B}
\tkzDefPoint(0.0,0.0){C}
\tkzDrawArc[<-,line width=0.9mm, blue](C,B)(A)
\tkzDefPoint(68.55786,-72.79986){A}
\tkzDefPoint(-68.55786,-72.79986){B}
\tkzDefPoint(0.0,0.0){C}
\tkzDrawArc[->,line width=0.9mm, red](C,B)(A)
\tkzDefPoint(-72.79986,-68.55786){A}
\tkzDefPoint(-72.79986,68.55786){B}
\tkzDefPoint(0.0,0.0){C}
\tkzDrawArc[->,line width=0.9mm, blue](C,B)(A)
\node [black,circle,draw,fill=white,scale=0.75,line width=1mm] (17) at (-70.71068,-70.71068) {};
\node [black,circle,draw,fill=white,scale=0.75,line width=1mm] (25) at (-70.71068,70.71068) {};
\node [black,circle,draw,fill=white,scale=0.75,line width=1mm] (26) at (70.71068,70.71068) {};
\node [black,circle,draw,fill=white,scale=0.75,line width=1mm] (27) at (70.71068,-70.71068) {};
\end{tikzpicture}

%% file: 7_seed_chosen_colour.tikz
\begin{tikzpicture}[scale=0.065]
\def\vertexscale{1.00}
\def\labelscale{1.20}
\node [circle,red,draw,scale=\vertexscale] (1) at (-35.04807,18.13716) {s=1};
\node [circle,black,draw,scale=\vertexscale] (2) at (16.85007,38.37066) {2};
\node [circle,black,draw,scale=\vertexscale] (3) at (7.90152,59.13273) {3};
\node [circle,black,draw,scale=\vertexscale] (4) at (-21.21792,-70.27781) {4};
\node [circle,blue,draw,scale=\vertexscale] (5) at (-5.04854,-49.67057) {5};
\node [circle,black,draw,scale=\vertexscale] (6) at (19.27363,-65.38332) {6};
\node [circle,black,draw,scale=\vertexscale] (7) at (70.12977,-31.96580) {7};
\node [circle,black,draw,scale=\vertexscale] (8) at (57.08462,-8.75877) {8};
\node [circle,blue,draw,scale=\vertexscale] (9) at (32.03659,-24.37914) {9};
\node [circle,blue,draw,scale=\vertexscale] (10) at (38.93179,-50.58079) {10};
\node [circle,black,draw,scale=\vertexscale] (11) at (-67.39295,5.74440) {11};
\node [circle,black,draw,scale=\vertexscale] (12) at (55.91587,22.62318) {12};
\node [circle,blue,draw,scale=\vertexscale] (13) at (21.65484,14.90236) {13};
\node [circle,black,draw,scale=\vertexscale] (14) at (-2.11924,-19.10702) {14};
\node [circle,black,draw,scale=\vertexscale] (15) at (-36.59874,-23.93519) {15};
\node [circle,black,draw,scale=\vertexscale] (16) at (-31.51319,72.98481) {16};
\node [circle,blue,draw,scale=\vertexscale] (17) at (-27.63136,51.58083) {17};
\node [circle,blue,draw,scale=\vertexscale] (18) at (-5.90103,68.57550) {t=18};
\tkzDefPoint(79.33533,60.87614){19}
\tkzDefPoint(-38.26834,-92.38795){20}
\tkzDefPoint(60.87614,79.33533){21}
\tkzDefPoint(-19.50903,-98.07853){22}
\tkzDefPoint(0.00000,-100.00000){23}
\tkzDefPoint(19.50903,-98.07853){24}
\tkzDefPoint(-52.24986,85.26402){25}
\tkzDefPoint(-64.94480,76.04060){26}
\tkzDefPoint(-76.04060,64.94480){27}
\tkzDefPoint(-85.26402,52.24986){28}
\tkzDefPoint(-23.34454,97.23699){29}
\tkzDefPoint(-7.84591,99.69173){30}
\tkzDefPoint(7.84591,99.69173){31}
\tkzDefPoint(23.34454,97.23699){32}
\tkzDefPoint(-99.14449,-13.05262){33}
\tkzDefPoint(-99.14449,13.05262){34}
\tkzDefPoint(-52.24986,-85.26402){35}
\tkzDefPoint(-64.94480,-76.04060){36}
\tkzDefPoint(-76.04060,-64.94480){37}
\tkzDefPoint(-85.26402,-52.24986){38}
\tkzDefPoint(85.26402,-52.24986){39}
\tkzDefPoint(76.04060,-64.94480){40}
\tkzDefPoint(64.94480,-76.04060){41}
\tkzDefPoint(52.24986,-85.26402){42}
\tkzDefPoint(98.07853,19.50903){43}
\tkzDefPoint(100.00000,0.00000){44}
\tkzDefPoint(98.07853,-19.50903){45}
\tkzDefPoint(92.38795,38.26834){46}
\tkzDefPoint(-92.38795,-38.26834){47}
\tkzDefPoint(-38.26834,92.38795){48}
\tkzDefPoint(92.38795,-38.26834){49}
\tkzDefPoint(38.26834,-92.38795){50}
\tkzDefPoint(-92.38795,38.26834){51}
\tkzDefPoint(38.26834,92.38795){52}
\draw [black] (1) to (26);
\node [draw=none,fill=none,scale=\labelscale] () at (-68.19205,79.84263) {7};
\draw [black] (1) to (17);
\draw [black] (1) to (3);
\draw [black] (1) to (2);
\draw [black] (1) to (15);
\draw [black] (1) to (11);
\draw [black] (2) to (3);
\draw [black] (2) to (21);
\node [draw=none,fill=none,scale=\labelscale] () at (63.91995,83.30210) {6};
\draw [black] (2) to (13);
\draw [black] (3) to (32);
\node [draw=none,fill=none,scale=\labelscale] () at (24.51176,102.09884) {4};
\draw [black] (3) to (18);
\draw [black] (4) to (5);
\draw [black] (4) to (23);
\node [draw=none,fill=none,scale=\labelscale] () at (0.00000,-105.00000) {8};
\draw [black] (4) to (22);
\node [draw=none,fill=none,scale=\labelscale] () at (-20.48448,-102.98245) {12};
\draw [black] (4) to (35);
\node [draw=none,fill=none,scale=\labelscale] () at (-54.86235,-89.52722) {3};
\draw [black] (5) to (14);
\draw [black] (5) to (6);
\draw [black] (6) to (10);
\draw [black] (6) to (42);
\node [draw=none,fill=none,scale=\labelscale] () at (54.86235,-89.52722) {2};
\draw [black] (6) to (24);
\node [draw=none,fill=none,scale=\labelscale] () at (20.48448,-102.98245) {7};
\draw [black] (7) to (8);
\draw [black] (7) to (45);
\node [draw=none,fill=none,scale=\labelscale] () at (102.98245,-20.48448) {6};
\draw [black] (7) to (39);
\node [draw=none,fill=none,scale=\labelscale] () at (89.52722,-54.86235) {16};
\draw [black] (7) to (40);
\node [draw=none,fill=none,scale=\labelscale] () at (79.84263,-68.19205) {1};
\draw [black] (8) to (12);
\draw [black] (8) to (44);
\node [draw=none,fill=none,scale=\labelscale] () at (105.00000,0.00000) {4};
\draw [black] (8) to (9);
\draw [black] (9) to (10);
\draw [black] (9) to (14);
\draw [black] (10) to (41);
\node [draw=none,fill=none,scale=\labelscale] () at (68.19205,-79.84263) {11};
\draw [black] (11) to (38);
\node [draw=none,fill=none,scale=\labelscale] () at (-89.52722,-54.86235) {16};
\draw [black] (11) to (33);
\node [draw=none,fill=none,scale=\labelscale] () at (-104.10171,-13.70525) {12};
\draw [black] (11) to (27);
\node [draw=none,fill=none,scale=\labelscale] () at (-79.84263,68.19205) {10};
\draw [black] (12) to (13);
\draw [black] (12) to (19);
\node [draw=none,fill=none,scale=\labelscale] () at (83.30210,63.91995) {11};
\draw [black] (12) to (43);
\node [draw=none,fill=none,scale=\labelscale] () at (102.98245,20.48448) {4};
\draw [black] (13) to (14);
\draw [black] (14) to (15);
\draw [black] (15) to (36);
\node [draw=none,fill=none,scale=\labelscale] () at (-68.19205,-79.84263) {18};
\draw [black] (15) to (37);
\node [draw=none,fill=none,scale=\labelscale] () at (-79.84263,-68.19205) {16};
\draw [black] (16) to (17);
\draw [black] (16) to (25);
\node [draw=none,fill=none,scale=\labelscale] () at (-54.86235,89.52722) {7};
\draw [black] (16) to (29);
\node [draw=none,fill=none,scale=\labelscale] () at (-24.51176,102.09884) {11};
\draw [black] (16) to (30);
\node [draw=none,fill=none,scale=\labelscale] () at (-8.23821,104.67632) {15};
\draw [black] (17) to (18);
\draw [black] (18) to (31);
\node [draw=none,fill=none,scale=\labelscale] () at (8.23821,104.67632) {15};
\tkzDefPoint(-85.26402,52.24986){A}
\tkzDefPoint(-86.67200,30.69216){B}
\tkzDefPoint(-99.14449,13.05262){C}
\tkzCircumCenter(A,B,C)\tkzGetPoint{D}
\tkzDrawArc[black](D,C)(A)
\tkzDefPoint(93.49426,35.47990){A}
\tkzDefPoint(93.49426,-35.47990){B}
\tkzDefPoint(0.0,0.0){C}
\tkzDrawArc[<-,line width=0.9mm, red](C,B)(A)
\tkzDefPoint(91.19850,-41.02235){A}
\tkzDefPoint(41.02235,-91.19850){B}
\tkzDefPoint(0.0,0.0){C}
\tkzDrawArc[<-,line width=0.9mm, green](C,B)(A)
\tkzDefPoint(35.47990,-93.49426){A}
\tkzDefPoint(-35.47990,-93.49426){B}
\tkzDefPoint(0.0,0.0){C}
\tkzDrawArc[->,line width=0.9mm, red](C,B)(A)
\tkzDefPoint(-41.02235,-91.19850){A}
\tkzDefPoint(-91.19850,-41.02235){B}
\tkzDefPoint(0.0,0.0){C}
\tkzDrawArc[<-,line width=0.9mm, orange](C,B)(A)
\tkzDefPoint(-93.49426,-35.47990){A}
\tkzDefPoint(-93.49426,35.47990){B}
\tkzDefPoint(0.0,0.0){C}
\tkzDrawArc[<-,line width=0.9mm, blue](C,B)(A)
\tkzDefPoint(-91.19850,41.02235){A}
\tkzDefPoint(-41.02235,91.19850){B}
\tkzDefPoint(0.0,0.0){C}
\tkzDrawArc[->,line width=0.9mm, green](C,B)(A)
\tkzDefPoint(-35.47990,93.49426){A}
\tkzDefPoint(35.47990,93.49426){B}
\tkzDefPoint(0.0,0.0){C}
\tkzDrawArc[->,line width=0.9mm, orange](C,B)(A)
\tkzDefPoint(41.02235,91.19850){A}
\tkzDefPoint(91.19850,41.02235){B}
\tkzDefPoint(0.0,0.0){C}
\tkzDrawArc[->,line width=0.9mm, blue](C,B)(A)
\node [black,circle,draw,fill=white,scale=0.75,line width=1mm] (20) at (-38.26834,-92.38795) {};
\node [black,circle,draw,fill=white,scale=0.75,line width=1mm] (46) at (92.38795,38.26834) {};
\node [black,circle,draw,fill=white,scale=0.75,line width=1mm] (47) at (-92.38795,-38.26834) {};
\node [black,circle,draw,fill=white,scale=0.75,line width=1mm] (48) at (-38.26834,92.38795) {};
\node [black,circle,draw,fill=white,scale=0.75,line width=1mm] (49) at (92.38795,-38.26834) {};
\node [black,circle,draw,fill=white,scale=0.75,line width=1mm] (50) at (38.26834,-92.38795) {};
\node [black,circle,draw,fill=white,scale=0.75,line width=1mm] (51) at (-92.38795,38.26834) {};
\node [black,circle,draw,fill=white,scale=0.75,line width=1mm] (52) at (38.26834,92.38795) {};
\end{tikzpicture}

%% file: altseed_6_7_16.tikz
\begin{tikzpicture}[scale=0.065]
\def\vertexscale{1.00}
\def\labelscale{1.00}
\node [circle,red,draw,scale=\vertexscale] (1) at (-21.93655,-0.84880) {s=1};
\node [circle,black,draw,scale=\vertexscale] (2) at (-50.03016,30.63311) {2};
\node [circle,black,draw,scale=\vertexscale] (3) at (-33.70268,51.13316) {3};
\node [circle,black,draw,scale=\vertexscale] (4) at (67.33184,-16.33827) {4};
\node [circle,blue,draw,scale=\vertexscale] (5) at (21.24812,-69.89172) {5};
\node [circle,black,draw,scale=\vertexscale] (6) at (-69.16428,15.97351) {6};
\node [circle,black,draw,scale=\vertexscale] (7) at (-60.64509,-16.43889) {7};
\node [circle,black,draw,scale=\vertexscale] (8) at (64.54109,23.39152) {8};
\node [circle,blue,draw,scale=\vertexscale] (9) at (-19.15220,-69.94054) {9};
\node [circle,blue,draw,scale=\vertexscale] (10) at (31.56934,57.30163) {10};
\node [circle,black,draw,scale=\vertexscale] (11) at (15.46054,23.00597) {11};
\node [circle,black,draw,scale=\vertexscale] (12) at (35.73977,0.73803) {12};
\node [circle,blue,draw,scale=\vertexscale] (13) at (32.54155,-33.77063) {13};
\node [circle,black,draw,scale=\vertexscale] (14) at (2.47897,-44.96707) {14};
\node [circle,blue,draw,scale=\vertexscale] (15) at (-26.14997,-34.48309) {15};
\node [circle,orange,draw,scale=\vertexscale] (16) at (-7.21381,54.33470) {t=16};
\tkzDefPoint(79.33533,60.87614){17}
\tkzDefPoint(-38.26834,-92.38795){18}
\tkzDefPoint(60.87614,79.33533){19}
\tkzDefPoint(-13.05262,-99.14449){20}
\tkzDefPoint(13.05262,-99.14449){21}
\tkzDefPoint(-55.55702,83.14696){22}
\tkzDefPoint(-70.71068,70.71068){23}
\tkzDefPoint(-83.14696,55.55702){24}
\tkzDefPoint(-19.50903,98.07853){25}
\tkzDefPoint(-0.00000,100.00000){26}
\tkzDefPoint(19.50903,98.07853){27}
\tkzDefPoint(-99.14449,-13.05262){28}
\tkzDefPoint(-99.14449,13.05262){29}
\tkzDefPoint(-55.55702,-83.14696){30}
\tkzDefPoint(-70.71068,-70.71068){31}
\tkzDefPoint(-83.14696,-55.55702){32}
\tkzDefPoint(83.14696,-55.55702){33}
\tkzDefPoint(70.71068,-70.71068){34}
\tkzDefPoint(55.55702,-83.14696){35}
\tkzDefPoint(99.14449,13.05262){36}
\tkzDefPoint(99.14449,-13.05262){37}
\tkzDefPoint(92.38795,38.26834){38}
\tkzDefPoint(-92.38795,-38.26834){39}
\tkzDefPoint(-38.26834,92.38795){40}
\tkzDefPoint(92.38795,-38.26834){41}
\tkzDefPoint(38.26834,-92.38795){42}
\tkzDefPoint(-92.38795,38.26834){43}
\tkzDefPoint(38.26834,92.38795){44}
\draw [black] (1) to (7);
\draw [black] (1) to (2);
\draw [black] (1) to (3);
\draw [black] (1) to (16);
\draw [black] (1) to (11);
\draw [black] (1) to (15);
\draw [black] (2) to (3);
\draw [black] (2) to (6);
\draw [black] (2) to (23);
\node [draw=none,fill=none,scale=\labelscale] () at (-74.24621,74.24621) {13};
\draw [black] (3) to (22);
\node [draw=none,fill=none,scale=\labelscale] () at (-58.33487,87.30431) {4};
\draw [black] (3) to (16);
\draw [black] (4) to (37);
\node [draw=none,fill=none,scale=\labelscale] () at (104.10171,-13.70525) {5};
\draw [black] (4) to (33);
\node [draw=none,fill=none,scale=\labelscale] () at (87.30431,-58.33487) {3};
\draw [black] (4) to (12);
\draw [black] (4) to (8);
\draw [black] (5) to (21);
\node [draw=none,fill=none,scale=\labelscale] () at (13.70525,-104.10171) {4};
\draw [black] (5) to (14);
\draw [black] (5) to (35);
\node [draw=none,fill=none,scale=\labelscale] () at (58.33487,-87.30431) {6};
\draw [black] (6) to (24);
\node [draw=none,fill=none,scale=\labelscale] () at (-87.30431,58.33487) {5};
\draw [black] (6) to (7);
\draw [black] (6) to (29);
\node [draw=none,fill=none,scale=\labelscale] () at (-104.10171,13.70525) {10};
\draw [black] (7) to (28);
\node [draw=none,fill=none,scale=\labelscale] () at (-104.10171,-13.70525) {8};
\draw [black] (7) to (32);
\node [draw=none,fill=none,scale=\labelscale] () at (-87.30431,-58.33487) {16};
\draw [black] (8) to (12);
\draw [black] (8) to (17);
\node [draw=none,fill=none,scale=\labelscale] () at (83.30210,63.91995) {7};
\draw [black] (8) to (36);
\node [draw=none,fill=none,scale=\labelscale] () at (104.10171,13.70525) {9};
\draw [black] (9) to (30);
\node [draw=none,fill=none,scale=\labelscale] () at (-58.33487,-87.30431) {10};
\draw [black] (9) to (14);
\draw [black] (9) to (20);
\node [draw=none,fill=none,scale=\labelscale] () at (-13.70525,-104.10171) {8};
\draw [black] (10) to (11);
\draw [black] (10) to (27);
\node [draw=none,fill=none,scale=\labelscale] () at (20.48448,102.98245) {9};
\draw [black] (10) to (19);
\node [draw=none,fill=none,scale=\labelscale] () at (63.91995,83.30210) {6};
\draw [black] (11) to (16);
\draw [black] (11) to (12);
\draw [black] (12) to (13);
\draw [black] (13) to (34);
\node [draw=none,fill=none,scale=\labelscale] () at (74.24621,-74.24621) {2};
\draw [black] (13) to (14);
\draw [black] (14) to (15);
\draw [black] (15) to (31);
\node [draw=none,fill=none,scale=\labelscale] () at (-74.24621,-74.24621) {16};
\draw [black] (16) to (26);
\node [draw=none,fill=none,scale=\labelscale] () at (-0.00000,105.00000) {15};
\draw [black] (16) to (25);
\node [draw=none,fill=none,scale=\labelscale] () at (-20.48448,102.98245) {7};
\tkzDefPoint(93.49426,35.47990){A}
\tkzDefPoint(93.49426,-35.47990){B}
\tkzDefPoint(0.0,0.0){C}
\tkzDrawArc[<-,line width=0.9mm, red](C,B)(A)
\tkzDefPoint(91.19850,-41.02235){A}
\tkzDefPoint(41.02235,-91.19850){B}
\tkzDefPoint(0.0,0.0){C}
\tkzDrawArc[<-,line width=0.9mm, green](C,B)(A)
\tkzDefPoint(35.47990,-93.49426){A}
\tkzDefPoint(-35.47990,-93.49426){B}
\tkzDefPoint(0.0,0.0){C}
\tkzDrawArc[->,line width=0.9mm, red](C,B)(A)
\tkzDefPoint(-41.02235,-91.19850){A}
\tkzDefPoint(-91.19850,-41.02235){B}
\tkzDefPoint(0.0,0.0){C}
\tkzDrawArc[<-,line width=0.9mm, orange](C,B)(A)
\tkzDefPoint(-93.49426,-35.47990){A}
\tkzDefPoint(-93.49426,35.47990){B}
\tkzDefPoint(0.0,0.0){C}
\tkzDrawArc[<-,line width=0.9mm, blue](C,B)(A)
\tkzDefPoint(-91.19850,41.02235){A}
\tkzDefPoint(-41.02235,91.19850){B}
\tkzDefPoint(0.0,0.0){C}
\tkzDrawArc[->,line width=0.9mm, green](C,B)(A)
\tkzDefPoint(-35.47990,93.49426){A}
\tkzDefPoint(35.47990,93.49426){B}
\tkzDefPoint(0.0,0.0){C}
\tkzDrawArc[->,line width=0.9mm, orange](C,B)(A)
\tkzDefPoint(41.02235,91.19850){A}
\tkzDefPoint(91.19850,41.02235){B}
\tkzDefPoint(0.0,0.0){C}
\tkzDrawArc[->,line width=0.9mm, blue](C,B)(A)
\node [black,circle,draw,fill=white,scale=0.75,line width=1mm] (18) at (-38.26834,-92.38795) {};
\node [black,circle,draw,fill=white,scale=0.75,line width=1mm] (38) at (92.38795,38.26834) {};
\node [black,circle,draw,fill=white,scale=0.75,line width=1mm] (39) at (-92.38795,-38.26834) {};
\node [black,circle,draw,fill=white,scale=0.75,line width=1mm] (40) at (-38.26834,92.38795) {};
\node [black,circle,draw,fill=white,scale=0.75,line width=1mm] (41) at (92.38795,-38.26834) {};
\node [black,circle,draw,fill=white,scale=0.75,line width=1mm] (42) at (38.26834,-92.38795) {};
\node [black,circle,draw,fill=white,scale=0.75,line width=1mm] (43) at (-92.38795,38.26834) {};
\node [black,circle,draw,fill=white,scale=0.75,line width=1mm] (44) at (38.26834,92.38795) {};
\end{tikzpicture}

%% file: altseed_5_6_16_chosen.tikz
\begin{tikzpicture}[scale=0.06]
\def\vertexscale{1.20}
\def\labelscale{1.20}
\node [circle,orange,draw,scale=\vertexscale] (1) at (11.72745,-56.28262) {s=1};
\node [circle,orange,draw,scale=\vertexscale] (2) at (-12.26206,-22.20158) {2};
\node [circle,black,draw,scale=\vertexscale] (3) at (27.29077,-15.95033) {3};
\node [circle,black,draw,scale=\vertexscale] (4) at (40.95585,8.76305) {4};
\node [circle,blue,draw,scale=\vertexscale] (5) at (-65.35816,-7.23652) {5};
\node [circle,black,draw,scale=\vertexscale] (6) at (-49.83135,-33.45978) {6};
\node [circle,black,draw,scale=\vertexscale] (7) at (31.94288,66.68313) {7};
\node [circle,black,draw,scale=\vertexscale] (8) at (31.95259,40.54985) {8};
\node [circle,blue,draw,scale=\vertexscale] (9) at (64.45982,36.51846) {9};
\node [circle,black,draw,scale=\vertexscale] (10) at (-27.71357,-59.32759) {10};
\node [circle,black,draw,scale=\vertexscale] (11) at (-20.21467,63.89366) {11};
\node [circle,black,draw,scale=\vertexscale] (12) at (-2.30417,34.61914) {12};
\node [circle,blue,draw,scale=\vertexscale] (13) at (-23.19048,14.91315) {13};
\node [circle,black,draw,scale=\vertexscale] (14) at (-55.47366,24.88342) {14};
\node [circle,blue,draw,scale=\vertexscale] (15) at (41.39129,-56.35305) {15};
\node [circle,black,draw,scale=\vertexscale] (16) at (59.13398,-30.61516) {t=16};
\tkzDefPoint(-60.87614,79.33533){17}
\tkzDefPoint(-38.26834,-92.38795){18}
\tkzDefPoint(-79.33533,60.87614){19}
\tkzDefPoint(-19.50903,-98.07853){20}
\tkzDefPoint(0.00000,-100.00000){21}
\tkzDefPoint(19.50903,-98.07853){22}
\tkzDefPoint(-83.14696,-55.55702){23}
\tkzDefPoint(-70.71068,-70.71068){24}
\tkzDefPoint(-55.55702,-83.14696){25}
\tkzDefPoint(-98.07853,19.50903){26}
\tkzDefPoint(-100.00000,0.00000){27}
\tkzDefPoint(-98.07853,-19.50903){28}
\tkzDefPoint(79.33533,-60.87614){29}
\tkzDefPoint(60.87614,-79.33533){30}
\tkzDefPoint(98.07853,19.50903){31}
\tkzDefPoint(100.00000,0.00000){32}
\tkzDefPoint(98.07853,-19.50903){33}
\tkzDefPoint(55.55702,83.14696){34}
\tkzDefPoint(70.71068,70.71068){35}
\tkzDefPoint(83.14696,55.55702){36}
\tkzDefPoint(-19.50903,98.07853){37}
\tkzDefPoint(-0.00000,100.00000){38}
\tkzDefPoint(19.50903,98.07853){39}
\tkzDefPoint(-38.26834,92.38795){40}
\tkzDefPoint(92.38795,-38.26834){41}
\tkzDefPoint(-92.38795,-38.26834){42}
\tkzDefPoint(38.26834,92.38795){43}
\tkzDefPoint(38.26834,-92.38795){44}
\tkzDefPoint(-92.38795,38.26834){45}
\tkzDefPoint(92.38795,38.26834){46}
\draw [black] (1) to (15);
\draw [black] (1) to (22);
\node [draw=none,fill=none,scale=\labelscale] () at (20.48448,-102.98245) {7};
\draw [black] (1) to (21);
\node [draw=none,fill=none,scale=\labelscale] () at (0.00000,-105.00000) {11};
\draw [black] (1) to (2);
\draw [black] (1) to (3);
\draw [black] (2) to (3);
\draw [black] (2) to (10);
\draw [black] (2) to (6);
\draw [black] (2) to (13);
\draw [black] (3) to (4);
\draw [black] (3) to (16);
\draw [black] (4) to (32);
\node [draw=none,fill=none,scale=\labelscale] () at (105.00000,0.00000) {5};
\draw [black] (4) to (12);
\draw [black] (4) to (8);
\draw [black] (5) to (27);
\node [draw=none,fill=none,scale=\labelscale] () at (-105.00000,0.00000) {4};
\draw [black] (5) to (14);
\draw [black] (5) to (6);
\draw [black] (6) to (10);
\draw [black] (6) to (24);
\node [draw=none,fill=none,scale=\labelscale] () at (-74.24621,-74.24621) {7};
\draw [black] (7) to (39);
\node [draw=none,fill=none,scale=\labelscale] () at (20.48448,102.98245) {1};
\draw [black] (7) to (34);
\node [draw=none,fill=none,scale=\labelscale] () at (58.33487,87.30431) {16};
\draw [black] (7) to (35);
\node [draw=none,fill=none,scale=\labelscale] () at (74.24621,74.24621) {6};
\draw [black] (7) to (8);
\draw [black] (8) to (9);
\draw [black] (8) to (12);
\draw [black] (9) to (36);
\node [draw=none,fill=none,scale=\labelscale] () at (87.30431,58.33487) {10};
\draw [black] (9) to (31);
\node [draw=none,fill=none,scale=\labelscale] () at (102.98245,20.48448) {14};
\draw [black] (10) to (20);
\node [draw=none,fill=none,scale=\labelscale] () at (-20.48448,-102.98245) {11};
\draw [black] (10) to (25);
\node [draw=none,fill=none,scale=\labelscale] () at (-58.33487,-87.30431) {9};
\draw [black] (11) to (17);
\node [draw=none,fill=none,scale=\labelscale] () at (-63.91995,83.30210) {16};
\draw [black] (11) to (37);
\node [draw=none,fill=none,scale=\labelscale] () at (-20.48448,102.98245) {10};
\draw [black] (11) to (38);
\node [draw=none,fill=none,scale=\labelscale] () at (-0.00000,105.00000) {1};
\draw [black] (11) to (12);
\draw [black] (12) to (13);
\draw [black] (13) to (14);
\draw [black] (14) to (26);
\node [draw=none,fill=none,scale=\labelscale] () at (-102.98245,20.48448) {9};
\draw [black] (14) to (19);
\node [draw=none,fill=none,scale=\labelscale] () at (-83.30210,63.91995) {15};
\draw [black] (15) to (30);
\node [draw=none,fill=none,scale=\labelscale] () at (63.91995,-83.30210) {14};
\draw [black] (15) to (16);
\draw [black] (16) to (33);
\node [draw=none,fill=none,scale=\labelscale] () at (102.98245,-20.48448) {7};
\draw [black] (16) to (29);
\node [draw=none,fill=none,scale=\labelscale] () at (83.30210,-63.91995) {11};
\tkzDefPoint(-83.14696,-55.55702){A}
\tkzDefPoint(-85.17598,-35.28105){B}
\tkzDefPoint(-98.07853,-19.50903){C}
\tkzCircumCenter(A,B,C)\tkzGetPoint{D}
\tkzDrawArc[black](D,A)(C)
\tkzDefPoint(-35.47990,93.49426){A}
\tkzDefPoint(35.47990,93.49426){B}
\tkzDefPoint(0.0,0.0){C}
\tkzDrawArc[<-,line width=0.9mm, red](C,B)(A)
\tkzDefPoint(41.02235,91.19850){A}
\tkzDefPoint(91.19850,41.02235){B}
\tkzDefPoint(0.0,0.0){C}
\tkzDrawArc[<-,line width=0.9mm, green](C,B)(A)
\tkzDefPoint(93.49426,35.47990){A}
\tkzDefPoint(93.49426,-35.47990){B}
\tkzDefPoint(0.0,0.0){C}
\tkzDrawArc[<-,line width=0.9mm, orange](C,B)(A)
\tkzDefPoint(91.19850,-41.02235){A}
\tkzDefPoint(41.02235,-91.19850){B}
\tkzDefPoint(0.0,0.0){C}
\tkzDrawArc[<-,line width=0.9mm, blue](C,B)(A)
\tkzDefPoint(35.47990,-93.49426){A}
\tkzDefPoint(-35.47990,-93.49426){B}
\tkzDefPoint(0.0,0.0){C}
\tkzDrawArc[->,line width=0.9mm, red](C,B)(A)
\tkzDefPoint(-41.02235,-91.19850){A}
\tkzDefPoint(-91.19850,-41.02235){B}
\tkzDefPoint(0.0,0.0){C}
\tkzDrawArc[->,line width=0.9mm, green](C,B)(A)
\tkzDefPoint(-93.49426,-35.47990){A}
\tkzDefPoint(-93.49426,35.47990){B}
\tkzDefPoint(0.0,0.0){C}
\tkzDrawArc[->,line width=0.9mm, orange](C,B)(A)
\tkzDefPoint(-91.19850,41.02235){A}
\tkzDefPoint(-41.02235,91.19850){B}
\tkzDefPoint(0.0,0.0){C}
\tkzDrawArc[->,line width=0.9mm, blue](C,B)(A)
\node [black,circle,draw,fill=white,scale=0.75,line width=1mm] (18) at (-38.26834,-92.38795) {};
\node [black,circle,draw,fill=white,scale=0.75,line width=1mm] (40) at (-38.26834,92.38795) {};
\node [black,circle,draw,fill=white,scale=0.75,line width=1mm] (41) at (92.38795,-38.26834) {};
\node [black,circle,draw,fill=white,scale=0.75,line width=1mm] (42) at (-92.38795,-38.26834) {};
\node [black,circle,draw,fill=white,scale=0.75,line width=1mm] (43) at (38.26834,92.38795) {};
\node [black,circle,draw,fill=white,scale=0.75,line width=1mm] (44) at (38.26834,-92.38795) {};
\node [black,circle,draw,fill=white,scale=0.75,line width=1mm] (45) at (-92.38795,38.26834) {};
\node [black,circle,draw,fill=white,scale=0.75,line width=1mm] (46) at (92.38795,38.26834) {};
\end{tikzpicture}

%% file: minseed10.tikz
\begin{tikzpicture}[scale=0.065]
\def\vertexscale{1.50}
\def\labelscale{1.70}
\node [circle,blue,draw,scale=\vertexscale] (1) at (30.26301,49.30159) {s=1};
\node [circle,black,draw,scale=\vertexscale] (2) at (53.32394,-1.86101) {2};
\node [circle,blue,draw,scale=\vertexscale] (3) at (27.86973,-48.78260) {3};
\node [circle,blue,draw,scale=\vertexscale] (4) at (-16.73279,-54.40718) {4};
\node [circle,blue,draw,scale=\vertexscale] (5) at (-18.39832,-22.81958) {5};
\node [circle,blue,draw,scale=\vertexscale] (6) at (8.28683,1.13455) {6};
\node [circle,blue,draw,scale=\vertexscale] (7) at (-0.76292,36.06687) {7};
\node [circle,blue,draw,scale=\vertexscale] (8) at (-30.66531,48.04957) {8};
\node [circle,blue,draw,scale=\vertexscale] (9) at (-60.68963,6.59584) {9};
\node [circle,teal,draw,scale=\vertexscale] (10) at (-40.81150,-11.66574) {t=10};
\tkzDefPoint(-100.00000,-0.00000){11}
\tkzDefPoint(-70.71068,-70.71068){12}
\tkzDefPoint(-25.88190,-96.59258){13}
\tkzDefPoint(25.88190,-96.59258){14}
\tkzDefPoint(100.00000,0.00000){15}
\tkzDefPoint(-25.88190,96.59258){16}
\tkzDefPoint(25.88190,96.59258){17}
\tkzDefPoint(-70.71068,70.71068){18}
\tkzDefPoint(70.71068,70.71068){19}
\tkzDefPoint(70.71068,-70.71068){20}
\draw [black] (1) to (17);
\node [draw=none,fill=none,scale=\labelscale] () at (27.17600,101.42221) {3};
\draw [black] (1) to (2);
\draw [black] (1) to (7);
\draw [black] (2) to (15);
\node [draw=none,fill=none,scale=\labelscale] () at (105.00000,0.00000) {9};
\draw [black] (2) to (3);
\draw [black] (2) to (6);
\draw [black] (3) to (14);
\node [draw=none,fill=none,scale=\labelscale] () at (27.17600,-101.42221) {1};
\draw [black] (3) to (4);
\draw [black] (4) to (5);
\draw [black] (4) to (13);
\node [draw=none,fill=none,scale=\labelscale] () at (-27.17600,-101.42221) {8};
\draw [black] (5) to (10);
\draw [black] (5) to (6);
\draw [black] (6) to (7);
\draw [black] (7) to (8);
\draw [black] (8) to (9);
\draw [black] (8) to (16);
\node [draw=none,fill=none,scale=\labelscale] () at (-27.17600,101.42221) {4};
\draw [black] (9) to (10);
\draw [black] (9) to (11);
\node [draw=none,fill=none,scale=\labelscale] () at (-105.00000,-0.00000) {2};
\tkzDefPoint(-68.55786,72.79986){A}
\tkzDefPoint(68.55786,72.79986){B}
\tkzDefPoint(0.0,0.0){C}
\tkzDrawArc[<-,line width=0.9mm, red](C,B)(A)
\tkzDefPoint(72.79986,68.55786){A}
\tkzDefPoint(72.79986,-68.55786){B}
\tkzDefPoint(0.0,0.0){C}
\tkzDrawArc[<-,line width=0.9mm, blue](C,B)(A)
\tkzDefPoint(68.55786,-72.79986){A}
\tkzDefPoint(-68.55786,-72.79986){B}
\tkzDefPoint(0.0,0.0){C}
\tkzDrawArc[->,line width=0.9mm, red](C,B)(A)
\tkzDefPoint(-72.79986,-68.55786){A}
\tkzDefPoint(-72.79986,68.55786){B}
\tkzDefPoint(0.0,0.0){C}
\tkzDrawArc[->,line width=0.9mm, blue](C,B)(A)
\node [black,circle,draw,fill=white,scale=0.75,line width=1mm] (12) at (-70.71068,-70.71068) {};
\node [black,circle,draw,fill=white,scale=0.75,line width=1mm] (18) at (-70.71068,70.71068) {};
\node [black,circle,draw,fill=white,scale=0.75,line width=1mm] (19) at (70.71068,70.71068) {};
\node [black,circle,draw,fill=white,scale=0.75,line width=1mm] (20) at (70.71068,-70.71068) {};
\end{tikzpicture}

%% file: 8and9seed_14.tikz
\begin{tikzpicture}[scale=0.065]
\def\vertexscale{1.00}
\def\labelscale{1.40}
\node [circle,red,draw,scale=\vertexscale] (1) at (7.86037,47.14605) {s=1};
\node [circle,black,draw,scale=\vertexscale] (2) at (28.95763,0.85338) {2};
\node [circle,black,draw,scale=\vertexscale] (3) at (-16.80373,3.73974) {3};
\node [circle,blue,draw,scale=\vertexscale] (4) at (-17.23258,-25.97814) {4};
\node [circle,blue,draw,scale=\vertexscale] (5) at (-47.07586,-27.42150) {5};
\node [circle,blue,draw,scale=\vertexscale] (6) at (59.69440,12.17010) {6};
\node [circle,black,draw,scale=\vertexscale] (7) at (45.81260,48.78440) {7};
\node [circle,blue,draw,scale=\vertexscale] (8) at (20.99309,-52.13774) {8};
\node [circle,blue,draw,scale=\vertexscale] (9) at (47.93027,-30.34245) {9};
\node [circle,black,draw,scale=\vertexscale] (10) at (-37.31577,-58.94222) {10};
\node [circle,black,draw,scale=\vertexscale] (11) at (-30.04170,66.61528) {11};
\node [circle,blue,draw,scale=\vertexscale] (12) at (-18.65141,52.15085) {12};
\node [circle,black,draw,scale=\vertexscale] (13) at (-33.79733,42.69617) {13};
\node [circle,blue,draw,scale=\vertexscale] (14) at (-53.49716,26.21880) {t=14};
\tkzDefPoint(-92.38795,38.26834){15}
\tkzDefPoint(-70.71068,-70.71068){16}
\tkzDefPoint(-100.00000,-0.00000){17}
\tkzDefPoint(-92.38795,-38.26834){18}
\tkzDefPoint(-38.26834,-92.38795){19}
\tkzDefPoint(0.00000,-100.00000){20}
\tkzDefPoint(38.26834,-92.38795){21}
\tkzDefPoint(92.38795,38.26834){22}
\tkzDefPoint(100.00000,0.00000){23}
\tkzDefPoint(92.38795,-38.26834){24}
\tkzDefPoint(-38.26834,92.38795){25}
\tkzDefPoint(-0.00000,100.00000){26}
\tkzDefPoint(38.26834,92.38795){27}
\tkzDefPoint(-70.71068,70.71068){28}
\tkzDefPoint(70.71068,70.71068){29}
\tkzDefPoint(70.71068,-70.71068){30}
\draw [black] (1) to (3);
\draw [black] (1) to (13);
\draw [black] (1) to (12);
\draw [black] (1) to (11);
\draw [black] (1) to (26);
\node [draw=none,fill=none,scale=\labelscale] () at (-0.00000,105.00000) {10};
\draw [black] (1) to (7);
\draw [black] (1) to (2);
\draw [black] (2) to (6);
\draw [black] (2) to (9);
\draw [black] (2) to (3);
\draw [black] (3) to (4);
\draw [black] (3) to (14);
\draw [black] (4) to (8);
\draw [black] (4) to (5);
\draw [black] (5) to (10);
\draw [black] (5) to (17);
\node [draw=none,fill=none,scale=\labelscale] () at (-105.00000,-0.00000) {6};
\draw [black] (6) to (7);
\draw [black] (6) to (23);
\node [draw=none,fill=none,scale=\labelscale] () at (105.00000,0.00000) {5};
\draw [black] (7) to (27);
\node [draw=none,fill=none,scale=\labelscale] () at (40.18176,97.00735) {8};
\draw [black] (7) to (22);
\node [draw=none,fill=none,scale=\labelscale] () at (97.00735,40.18176) {14};
\draw [black] (8) to (9);
\draw [black] (8) to (21);
\node [draw=none,fill=none,scale=\labelscale] () at (40.18176,-97.00735) {7};
\draw [black] (9) to (24);
\node [draw=none,fill=none,scale=\labelscale] () at (97.00735,-40.18176) {10};
\draw [black] (10) to (19);
\node [draw=none,fill=none,scale=\labelscale] () at (-40.18176,-97.00735) {11};
\draw [black] (10) to (18);
\node [draw=none,fill=none,scale=\labelscale] () at (-97.00735,-40.18176) {9};
\draw [black] (10) to (20);
\node [draw=none,fill=none,scale=\labelscale] () at (0.00000,-105.00000) {1};
\draw [black] (11) to (13);
\draw [black] (11) to (25);
\node [draw=none,fill=none,scale=\labelscale] () at (-40.18176,97.00735) {10};
\draw [black] (11) to (12);
\draw [black] (12) to (13);
\draw [black] (13) to (14);
\draw [black] (14) to (15);
\node [draw=none,fill=none,scale=\labelscale] () at (-97.00735,40.18176) {7};
\tkzDefPoint(-68.55786,72.79986){A}
\tkzDefPoint(68.55786,72.79986){B}
\tkzDefPoint(0.0,0.0){C}
\tkzDrawArc[<-,line width=0.9mm, red](C,B)(A)
\tkzDefPoint(72.79986,68.55786){A}
\tkzDefPoint(72.79986,-68.55786){B}
\tkzDefPoint(0.0,0.0){C}
\tkzDrawArc[<-,line width=0.9mm, blue](C,B)(A)
\tkzDefPoint(68.55786,-72.79986){A}
\tkzDefPoint(-68.55786,-72.79986){B}
\tkzDefPoint(0.0,0.0){C}
\tkzDrawArc[->,line width=0.9mm, red](C,B)(A)
\tkzDefPoint(-72.79986,-68.55786){A}
\tkzDefPoint(-72.79986,68.55786){B}
\tkzDefPoint(0.0,0.0){C}
\tkzDrawArc[->,line width=0.9mm, blue](C,B)(A)
\node [black,circle,draw,fill=white,scale=0.75,line width=1mm] (16) at (-70.71068,-70.71068) {};
\node [black,circle,draw,fill=white,scale=0.75,line width=1mm] (28) at (-70.71068,70.71068) {};
\node [black,circle,draw,fill=white,scale=0.75,line width=1mm] (29) at (70.71068,70.71068) {};
\node [black,circle,draw,fill=white,scale=0.75,line width=1mm] (30) at (70.71068,-70.71068) {};
\end{tikzpicture}